\documentclass[a4paper,11pt, reqno]{amsart}

\usepackage{amsmath,amsfonts,amssymb,graphics}
\usepackage{amsthm}
\usepackage{pgfplots}
\usepackage{bbm,tikz}
\usepackage{cite}
\usepackage{hyperref}
\usepackage{color}
\usepackage[normalem]{ulem}
\usepackage{enumitem}
\usepackage[left=1.2in,right=1.2in,top=1.3in,bottom=1.3in]{geometry}
\usepackage[foot]{amsaddr}

\newcommand \id{\mathbbm 1}

\newtheorem{theorem}{Theorem}[section]
\newtheorem{lemma}[theorem]{Lemma}
\newtheorem{proposition}[theorem]{Proposition}
\newtheorem{corollary}[theorem]{Corollary}

\newtheorem{conjecture}[theorem]{Conjecture}

\theoremstyle{remark}
\newtheorem{remark}[theorem]{Remark}

\numberwithin{equation}{section}

\begin{document}

\title[The band structure of a model of spatial random permutation]{The band structure of a model \\ of spatial random permutation}

\author{Yan V. Fyodorov$^1$}
\address{$^1$Department of Mathematics, King's College London}
\email{yan.fyodorov@kcl.ac.uk}
\author{Stephen Muirhead$^2$}
\email{smui@unimelb.edu.au}
\address{$^2$Department of Mathematics, King's College London (current address: School of Mathematics and Statistics, University of Melbourne)}
\date{\today}
\thanks{This research was supported by the Engineering and Physical Sciences Research Council (EPSRC) Grant EP/N009436/1 ``The many faces of random characteristic polynomials'' and the Australian Research Council (ARC) Discovery Early Career Researcher Award DE200101467. The authors would like to thank Jeremiah Buckley, Naomi Feldheim and Daniel Ueltschi for enlightening discussions, and in particular Ron Peled for helpful discussions at an early stage. The authors would also like to thank an anonymous referee for detailed comments which improved the presentation of the paper, and also for pointing out corrections to an earlier version.}

\begin{abstract} 
We study a random permutation of a lattice box in which each permutation is given a Boltzmann weight with energy equal to the total Euclidean displacement. Our main result establishes the \textit{band structure} of the model as the box-size $N$ tends to infinity and the inverse temperature~$\beta$ tends to zero; in particular, we show that the mean displacement is of order $\min \{ 1/\beta, N\}$. In one dimension our results are more precise, specifying leading-order constants and giving bounds on the rates of convergence. 

Our proofs exploit a connection, via matrix permanents, between random permutations and Gaussian fields; although this connection is well-known in other settings, to the best of our knowledge its application to the study of random permutations is novel. As a byproduct of our analysis, we also provide asymptotics for the permanents of Kac-Murdock-Szeg\H{o}~(KMS) matrices.
\end{abstract}
\maketitle


\section{Introduction}

A \textit{spatial random permutation} (SRP) is a probability measure on a set of permutations which is biased towards the identity in some underlying geometry. Among the most well-studied models of SRP is the \textit{Mallows model} \cite{M57} in which each permutation $\pi$ of the set $[\![1, N]\!] := \{1, 2, \ldots, N\}$ is assigned weight
\[  P(\pi)  \propto  q^{\text{inv}(\pi)} , \]
where $q \in (0, 1]$ is a parameter and 
\[ \text{inv}(\pi) := | \{  (s, t) : s < t \ \ \text{and} \ \ \pi(s) > \pi(t)  \} | \]
is the \textit{inversion count} of $\pi$; as is well-known, $\text{inv}(\pi)$ equals the minimal number of adjacent transpositions required to bring $\pi$ to the identity, and so the Mallows model can be considered as a SRP arising from the \textit{Kendall tau metric}
\[ d( \pi_1, \pi_2) := \text{minimal number of adjacent transpositions required to bring } \pi_1 \text{ to } \pi_2 .  \]
 In the case $q = 1$, the Mallows model reduces to the classical model of \textit{uniform random permutation}.

\smallskip
The Mallows model is a particularly tractable SRP because it possesses a certain integrable structure \cite{M57, GP16}: conditionally on observing the partial mapping $(\pi(i))_{ 1 \le i  \le k}$, the value of $\pi(k+1)$ is distributed geometrically on the ordered set of remaining sites $[\![1, N ]\!] \setminus \{\pi(i): 1 \le i \le  k\}$. This Markov-type property greatly facilitates computations, and many statistical properties of the Mallows model have recently been derived, for instance the distribution of the longest increasing subsequence \cite{BP15, MS14}, and detailed information on the cycle structure \cite{GP16, M16b}.

\smallskip
A basic statistical property of particular importance is the \textit{band structure} of the Mallows model: if one plots the graph of $(i, \pi(i))$, the majority of points lie inside a strip centred on the diagonal. Moreover, the width of the strip exhibits \textit{crossover} behaviour: as $q \to 1$ and $N \to \infty$, the width is the minimum of $1/(1-q)$ and~$N$, up to leading-order constants. For example, if one considers just the  \textit{mean displacement} $E[|i-\pi(i)|]$, then it is known \cite{BP15} that there exists a constant $c > 0$ such that, for each $i \in [\![1, N]\!]$,
\begin{equation}
\label{e:band}
 c \, \text{min} \left\{ \frac{q}{1-q}, N-1 \right\} \le  E [  |i - \pi(i) | ]   \le \min \left\{  \frac{2q}{1-q} , N- 1 \right\} .   
 \end{equation}
More precise concentration bounds on the displacement $|i - \pi(i)|$ are also available, see \cite{GP16, M16a}. 

\smallskip
 It has been conjectured \cite{GP16} that many different models of SRPs possess similar statistical features to the Mallows model. A more general class of SRPs are the \textit{Boltzmann SRPs} 
 \begin{equation}
\label{e:bsrp}
   P(\pi) \propto e^{-\beta H(\pi) }  ,
   \end{equation}
where $\beta \ge 0$ is an \textit{inverse temperature} parameter and $H(\pi)$ is an \textit{energy function} that depends on the distance from the identity in an underlying geometry; in the infinite-temperature limit $\beta = 0$ this reduces to the uniform random permutation. An important subclass of Boltzmann SRPs -- this time not containing the Mallows model -- arises when considering random bijections on a set of Euclidean particles. To define this subclass, let $\mathcal{T}$ be a finite set of points in $\mathbb{R}^d$ and let $V : \mathbb{R}^d \to \mathbb{R}^+$ be a \textit{potential function}. To each bijection $\pi$ on the set $\mathcal{T}$ one can associate an energy function
\[ H(\pi) := \sum_{x \in \mathcal{T}} V(x - \pi(x) ) , \]
which defines a model of random bijection via the measure \eqref{e:bsrp}; we shall refer to this class of models as \textit{random Euclidean bijections}. Natural choices for the potential include $V(x) = |x|$ and $V(x) = |x|^2$, where $|\cdot|$ denotes the standard Euclidean distance on~$\mathbb{R}^d$; the latter choice is particularly important since it is connected to the classical representation of the Bose gas (see Section \ref{s:lr} below).

\smallskip
In this paper we consider the random Euclidean bijection in which the particles are a subset of the lattice $\mathbb{Z}^d$ and the potential is the \textit{Euclidean distance} $V(x) = |x|$. This model has been considered previously in the literature, usually within the setting of general potential functions $V$ and arbitrary rescaled versions of the lattice $\mathbb{Z}^d$, see, e.g., \cite{B14, BR15, GRU07, M16a}. We focus on this choice of potential because, as we explain in Section \ref{s:h}, the resulting model possesses a tractable structure which, to the best of our knowledge, has not yet been identified in the literature. Let us introduce the model formally now.

\smallskip
Define the size-$N$ lattice box $\mathcal{T}^N :=  [\![1, N]\!]^d \subset \mathbb{Z}^d$. Let $\pi^N$ denote the set of permutations of the set $\mathcal{T}^N$, i.e., the set of bijections $\pi: \mathcal{T}^N \to \mathcal{T}^N$. For each $\pi \in \pi^N$, define the energy
 \begin{equation}
 \label{e:en}
  H^N(\pi) := \sum_{x \in \mathcal{T}^N} |x - \pi(x) | .
  \end{equation}
 For each $N\ge 1$ and $\beta \ge 0$, denote by $P_\beta^N$ the Boltzmann distribution associated to $H^N$ with inverse temperature $\beta$:
 \begin{equation}
\label{e:pn}
P_{\beta}^{N}(\pi) := \frac{1}{Z_\beta^N} \, e^{-\beta H^N(\pi) }  ,
\end{equation}
where $Z_\beta^N$ denotes the \textit{partition function}
\begin{equation}
\label{e:zn}
 Z^N_\beta :=   \sum_{\pi \in \pi^N} e^{-\beta H^N(\pi) }  =  \sum_{\pi \in \pi^N} e^{-\beta\sum_{x \in \mathcal{T}^N} |x-\pi(x)|}  . 
\end{equation}

\smallskip
We are interested in the statistical properties of the random permutation $P_\beta^N$ as $N \to \infty$ and $\beta \to 0$ at a particular rate. Observe that the regime $\beta N^{d+1} \to 0$ is trivial, since for all bijections $\pi$
\[ \beta H^N(\pi) = \beta \sum_{x \in \mathcal{T}^N} |x - \pi(x) | \le \beta |\mathcal{T}^N| \, \text{diam}(\mathcal{T}^N)  \le \sqrt{d} \beta N^{d+1}  , \]
and so if $\beta N^{d+1} \to 0$ the model converges to the uniform random permutation on $\mathcal{T}^N$ in the sense that
\[ \max_{\pi} | P_{\beta}^{N}(\pi)   / P_0^N(\pi)  - 1|  \to 0 . \]
On the other hand, the behaviour of this model in other regimes is rather complex, depending on a delicate balance between energy and entropy. Further, unlike for the Mallows model, as far as we know there is no integrable structure to exploit.

\smallskip
Our primary focus is on the \textit{band structure} of the model, analogous to~\eqref{e:band}. To this end, define the \textit{mean displacement per site}
\[   D^N_\beta :=  \frac{1}{N^d}  \sum_{x \in \mathcal{T}^N}  E[|x - \pi(x)|]  = \frac{1}{N^d}  E [  H^N(\pi) ] , \]
where $E[\cdot]$ is the expectation operator associated to the measure $P^N_\beta$ (note that we have dropped the explicit dependence on $N$ and $\beta$). It is not hard to see that $D^N_\beta$ is a non-increasing function of $\beta$. Our main result establishes the asymptotic growth-rate of $D^N_\beta$ as $\beta \to 0$ and $N \to \infty$. Similarly to in \eqref{e:band}, we observe \textit{crossover behaviour} for the band-width; the growth rate of $D^N_\beta$ is the minimum of  $ 1/ \beta$ and $ N$, up to leading-order constants. 
 
 \smallskip
 To state our main result precisely, let us first introduce some asymptotic notation. The inverse temperature parameter $\beta = \beta(N)$ will always be implicitly varying with~$N$. For two functions $f = f(N)$ and $g = g(N)$ we write $f \ll g$ or $f = o(g)$ if $|f|/|g| \to 0$ as $N \to \infty$, and $f \gg g$ or $g = o(f)$ if $|f|/|g| \to \infty$ as $N \to \infty$. Similarly, we write $f \sim g$ if $f/g \to 1$ as $N \to \infty$. Finally, we write $f = O(g)$ if there exists a $c > 0$ such that $|f|/|g| < c$ for sufficiently large $N$, and $f = \Theta(g)$ if $f = O(g)$ and $g = O(f)$ both hold.
 
 \smallskip
Our first main result establishes the band structure of the model in all dimensions:
 
\begin{theorem}[Band structure of the model]
\label{t:main1}
Let $d \ge 1$. Then, as $N \to \infty$,
\[ D^N_\beta  = \begin{cases}
 \Theta(1/\beta),  & \text{ if } 1/N \ll \beta \ll 1 , \\
    \Theta(N)   , & \text{ if } \beta =O(1/N) . \\
  \end{cases}
    \]
  Moreover, there exists a non-increasing function $f^d : \mathbb{R}^+ \to \mathbb{R}^+$ and a countable set $\mathcal{C}^d \subset \mathbb{R}^+$ such that if $c \in \mathbb{R}^+ \setminus \mathcal{C}^d$ and $\beta = c/N$ then, as $N \to \infty$,
    \[ D^N_\beta =  N f^d(c) \times (1 + o(1) ) .    \]
\end{theorem}

In dimension one we describe the band structure in a more precise way, providing leading-order constants and rates of convergence:

 \begin{theorem}[Band structure of the model in dimension one]
\label{t:main2}
Let $d = 1$. Then there exists a smooth non-increasing function $f^1 : \mathbb{R}^+ \to \mathbb{R}^+$ such that, as $N \to \infty$,
\[ D^N_\beta = \begin{cases}
   1/\beta \times \left( 1 +  O \left( \beta^{1/2}   \right) + O\left( (\beta N)^{-1/2} \right)  \right) ,  &  \text{ if }  1/N \ll \beta \ll 1, \\
     N f^1(c) \times ( 1 + o(1) ),  &   \text{ if }  \beta \sim c /N  , \ c > 0, \\
       N/3  \times  \left(1 + o(1)   \right)   , & \text{ if } \beta \ll 1/N . \\
  \end{cases}
    \]
    The function $f^1$ satisfies $\lim_{c \to 0} f^1(c) = 1/3$ and $f^1(c) = \Theta(1/c)$ as $c \to \infty$, and can be represented as
      \begin{equation}
   \label{e:frep}
    f^1(c) = - \frac{d}{dc} \left(  - h_c(0)^2 + \int_{s \in [0, c]}  \frac{2}{c}  \log h_c(s) - \frac{1}{2} \left( h_c(s) +  \dot{h}_c(s) \right)^2  \, ds \right)  
    \end{equation}
    where $h_c$ is the positive analytic function that is the unique solution to the boundary value ODE on~$[0, c]$
    \begin{equation}
    \label{e:ode}
     \ddot{y} =  -\frac{2}{c y}  +  y  \ , \quad \dot{y}(0) = y(0) = y(c) = -\dot{y}(c) .
     \end{equation}
\end{theorem}

\begin{remark}
We refer to the regime $\beta \sim c / N$ as \textit{critical}, and the regimes $1/N \ll \beta$ and $\beta \ll 1/N$ respectively as \textit{subcritical} and \textit{supercritical}. The critical regime for $d = 1$ was previously studied in \cite{M16a}, where the existence of the function $f^1$ was established and $f^1$ was shown to be \textit{continuous} and \textit{strictly decreasing}. Our analysis gives a new description of the function $f^1$ that, in addition, shows that it is \textit{smooth} and permits an analysis of its asymptotic behaviour (on the other hand, we are unable to deduce that $f^1$ is strictly decreasing from our description, only that it is non-increasing); see Figure~\ref{fig:f} for a rough illustration of this function. The remaining parts of Theorems \ref{t:main2}, and the entirety of Theorem~\ref{t:main1} in the case $d \ge 2$, are to the best of our knowledge completely new. We give more details on connections to the literature in Section~\ref{s:lr}. 
\end{remark}

\begin{remark}
\label{r:limit}
The exceptional set $\mathcal{C}^d$ arises in Theorem \ref{t:main1} because, in the case $d \ge 2$, we cannot rule out that $f^d$ has a countable set of jump discontinuities (see Remark \ref{r:except}). Similarly, we consider only $\beta = c/N$ rather than $\beta \sim c/N$ since this set could be dense. However, we believe $f^d$ to be smooth in all dimensions, which if true would allow the results in the critical regime to be extended to $\beta \sim c/N$ for all $c > 0$.
\end{remark}

\begin{remark}
\label{r:hyper}
In the uniform case, the mean displacement per site satisfies $ D^N_0 / N \to c^d$, where $c^d > 0$ is the \textit{hypercube line picking constant}
\[   c^d = \mathbb{E}[ |U_1-U_2|] , \]
where $U_i$ are independent random variables uniformly distributed over the hypercube $[0, 1]^d$. Integral representations for $c^d$ are known \cite{BBC10}, but no closed-form expression exists in general; on the other hand, a simple symmetry argument shows that $c^1 = 1/3$.

\smallskip
While our results in Theorem \ref{t:main2} confirm that $c^1  =  1/3 =  \lim_{c \to 0} f^1(c)$,  we are unable to show in general that
\[ c^d = \lim_{c \to 0} f^d(c) , \]
i.e.\ whether the mean displacement per site in the supercritical regime is necessarily first-order equivalent to the mean displacement in the uniform case:
\[   D^N_\beta / D^N_0 \to 1  \quad \text{if } \beta \ll 1/N .\]
\end{remark}

\begin{remark}
As explained in Section \ref{s:vf}, the representation of the function $f^1$ in \eqref{e:frep} arises out of a variational formula in the setting of large deviation theory for Gaussian fields. In \cite{M16a} the function $f^1$ was also represented via a variational formula that arose out of large deviation theory in a different setting (see Section \ref{s:lr} for details), and so one consequence of Theorem~\ref{t:main2} is an equivalence between two ostensibly unrelated variational formulae (see Proposition \ref{p:vf}).
\end{remark}

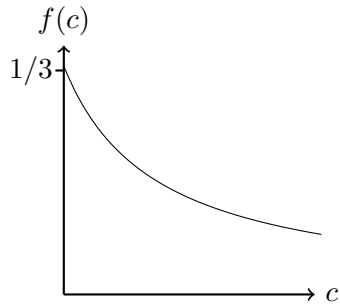
\begin{figure}[ht]
\begin{tikzpicture}[scale=1.1]	
\draw[->, thick] (0,0)--(3,0) node[right]{$c$};
\draw[->, thick] (0,0)--(0,3) node[above]{$f^1(c)$};
\draw (0,2.7) [-, thick] (-0.1,2.7)--(0,2.7) node[left]{$1/3$};
 \draw[scale=1.1,domain=0:2.8,smooth,variable=\x] plot ({\x},{2.5/(1+\x))});
\end{tikzpicture}
\caption{A rough illustration of the function $f^1$ in Theorem \ref{t:main2}.}
\label{fig:f}
\end{figure}

\begin{remark}
Our techniques are robust enough to apply to various modifications and generalisations of the model. First, instead of the lattice box $[\![1, N]\!]^d$ one could work instead in the more general setting of rescaled lattice domains $\mathbb{Z}^d \cap N D$, where $D \subset \mathbb{R}^d$ is an arbitrary smooth compact domain and $N D = \{ N x : x \in D \}$. In this setting we could recover the results of Theorems~\ref{t:main1} and \ref{t:main2} without change to the techniques. 

\smallskip 
Second, one could dispense with the requirement that the particles be confined to a lattice and instead work with \textit{disordered particles}. For example, prior to defining $P^N_\beta$ one could choose~$N^d$ points uniformly at random in the box $[0, N]^d$ and define $P^N_\beta$ by analogy to \eqref{e:pn}. Then, the results in Theorem \ref{t:main1}, as well as the (super-)critical regimes $\beta = O(1/N)$ in Theorem \ref{t:main2}, would still hold with probability tending to one since one can check that the relevant techniques are still valid for disordered points. On the other hand, the results in the subcritical regime $\beta \gg 1/N$ of Theorem \ref{t:main2} depend on precise equal spacing between the particles (see the analysis in Section~\ref{s:ou}), and so in the disordered case these results would not follow from our techniques. 
\end{remark}

\subsection{Connections to the literature}
\label{s:lr}
The model that we study has been considered previously in the literature, notably in \cite{B14, BR15, F91, GRU07, M16a, M16b}. In this section we give an account of this literature, and also discuss related results on a similar model in which the particles $x_i$ are not confined to a lattice, see, e.g., \cite{BU09, EP17}.

\smallskip
The most pertinent work is \cite{M16a}, which gave a detailed treatment of the one-dimensional model in the critical regime $\beta = cN$ using the concept of \textit{permuton limits}. Let us explain the content of this work, developed in a somewhat general setting. Consider a random permutation $P^N$ of the set $[\![1, N]\!]$. One way to encode the properties of $P^N$ is via the induced law on the set of empirical measures
\[ \nu(\pi) =  \frac{1}{N} \sum_{i=1}^N \delta_{ (i/N, \pi(i)/N ) } , \]
where $\delta_{(x,y)}$ denotes a unit $\delta$-mass at the point $(x, y)$ of the unit square $S$. Observe that every empirical measure $\nu(\pi)$ is a \textit{copula}, i.e., a probability measure on $S$ such that the marginals are uniform probability measures. Let $\mathcal{M}$ denote the space of copulas equipped with the Borel $\sigma$-algebra induced by the topology of weak convergence. A \textit{permuton} is a law (i.e.\ probability measure) on $\mathcal{M}$. Denote by $\mu^N$ the permuton induced by the empirical measures of $\pi$ under~$P^N$.

\smallskip
In \cite{M16a, M16b}, general sequences of SRPs are considered in which the associated permutons~$\mu^N$ converge weakly, as $N \to \infty$, to a permuton limit $\mu^\infty$. Let us describe the general setting of this result, and show how it implies Theorem~\ref{t:main2} in the case $\beta = cN$. Let $f$ be a continuous function on $S$. For each $N \ge 1$ and $c > 0$ consider the random permutation on $[\![1, N]\!]$ defined by
\begin{equation}
\label{e:mm}
P^N(\pi) \propto e^{c \sum_{i=1}^N f(i/N, \pi(i)/N )} , 
\end{equation}
with $\mu^N$ the associated permuton. Then the main result of \cite{M16a} is that, as $N \to \infty$, $\mu^N$ converges weakly to a permuton limit $\mu$, and moreover $\mu$ is supported on a single copula $\nu$ that depends on $f$ and $c$. As a consequence, one deduces that, as $N \to \infty$,
\[ \frac{1}{N} \sum_i^N f(i/N, \pi(i)/N)  \to \int_{(x, y) \in S} f(x, y) \, d \nu  (x, y) \quad \text{in probability}, \]
and in particular, letting $E^N$ denote the expectation operator associated to $P^N$,
\[   \frac{1}{N} \sum_i^N E^N f(i/N, \pi(i)/N)  \to   \int_{(x, y) \in S} f(x, y) \, d \nu (x, y) .\]
Specialising to the case that $f(x, y) = |x-y|$, one immediately deduces Theorem \ref{t:main2} in the case that $\beta = c/N$. 

\smallskip
Although not considered in \cite{M16a}, the result can likely also be adapted to cover the case $\beta \sim c/N$, which corresponds to defining the model \eqref{e:mm} with a varying sequence of parameters $c^N \to c$. It is also possible that the approach could be adapted to higher-dimensional analogues of \eqref{e:mm}, although additional complications may arise. Such a generalisation would give an alternative way to prove Theorem \ref{t:main1} in the regimes $\beta = O(1/N)$ and $\beta \sim c/N$ (and indeed should allow one to remove the exceptional set $\mathcal{C}^d$ in Theorem \ref{t:main1}).

\smallskip
The proof in \cite{M16a} relies on a large deviation principle, in the space $\mathcal{M}$, for the sequence of permutons associated with the uniform random permutation. Arising naturally out of this approach is a variational formula for the function $f$ in Theorem \ref{t:main2}. In particular, one deduces that $f(c) = \frac{d}{dc} V_P(c)$, where $V_P(c)$ is the solution to the variational problem in the space of permutons
\[  (VP\!:\!P) \qquad \text{maximise}  \quad  \frac{1}{c}   \int_{(x, y)\in S} |x-y| \, d \lambda - D(\lambda \| u )  \quad \text{over} \quad \lambda \in \mathcal{M}, \]
where $D(\cdot \| \cdot)$ is the Kullback-Leibler divergence and $u$ denotes the Lebesgue measure on the unit square. From this formula one readily deduces that $f$ is continuous and strictly increasing. We remark that $(VP\!:\!P)$ is quite different from the variational formula for $f$ generated by our approach (see Section \ref{s:vf}); on the other hand, both variational formulae ultimately arise out of large deviation principles, albeit in somewhat different settings.

\smallskip
In \cite{B14, BR15, F91, GRU07} the model in the case $\beta > 0$ fixed was considered, addressing in particular the question of whether the model has an infinite volume limit. In our setting, this is the question of whether one can define a measure $P_\beta^\infty$ on the set of bijections of $\mathbb{Z}^d$ such that the measures $P^N_\beta$ converge to $P_\beta^\infty$ in an appropriate sense as $N \to \infty$. As was shown in \cite{B14, BR15}, this is possible for any fixed $\beta > 0$ but not in the regime $\beta \to 0$. It is natural to expect that infinite volume limits are associated with the convergence of the mean displacement per site
\[ \lim_{N \to \infty} D^N_\beta = \bar c_\beta , \quad \bar c_\beta > 0 ,\]
but we do not pursue this connection here.

\smallskip
As mentioned in the introduction, random Euclidean bijections are of physical importance since they are related to classical representations of quantum gases. In particular, the ideal Bose gas is classically represented as a finite set of particles $\mathcal{T} = (x_i)_i$ in a compact set~$D \subset \mathbb{R}^d$ whose probability density with respect to the Lebesgue measure is given by
\[     f(x_1, \ldots, x_N ) \propto \sum_\pi  e^{- \sum_{i} | x_i - \pi(x_i) |^2 } ,   \]
where $\pi$ runs over all bijections of the particles $x_i$. This is the marginal of the measure on $(x_i; \pi)$ given by
\begin{equation}
\label{e:ann}
P(x_i; \pi) \propto e^{- \sum_i | x_i - \pi(x_i) |^2 } . 
\end{equation}
Feynman argued \cite{F53} that the occurrence of macroscopic cycles in \eqref{e:ann} is related to the onset of Bose-Einstein condensation in the ideal Bose gas; this was first shown mathematically forty years later \cite{S91, S02}. In \cite{BU09, EP17} the model was studied in the case of general potential
\[  P(x_i; \pi) \propto e^{- \sum_{i} V( x_i - \pi(x_i) ) } , \]
(the potential $V(x) = |x|$ is included in the class considered in \cite{BU09} at least for $d = 3$, but is not included in the class considered in \cite{EP17}). Note that particles in this model are \textit{not} confined to a lattice, and instead their positions influence the weighting of the ensemble; in this sense the ideal Bose gas can be considered as an \textit{annealed} version of the random Euclidean bijections we study in this paper. Remarkably, after averaging with respect to the positions the underlying permutation model turns out to possess an integrable structure that is not present when the points are fixed \cite{BU09}, which greatly facilitates the analysis of the ideal Bose gas.

\subsection{Overview of our analysis: Permutations, permanents and Gaussian fields} 
\label{s:h}
In this section we outline the central ideas of our analysis, which exploits a connection between SRPs and centred Gaussian fields, via the permanents of certain real symmetric matrices.

\smallskip
A preliminary observation is the relation between the mean displacement per site $D^N_\beta$ and the partition function $Z^N_\beta$ introduced in \eqref{e:zn}. In particular, we have the identity
\begin{align}
\label{e:s1}
\nonumber D^N_\beta &=  \frac{1}{N^d} \sum_{x \in \mathcal{T}^N}  E^N_\beta \left[  |x-\pi(x)| \right] = \frac{1}{N^d} \sum_{\pi \in \pi^N} P^N_\beta(\pi) \sum_{x \in \mathcal{T}^N}|x-\pi(x)|  \\
\nonumber & =\frac{1}{N^d} \frac{1}{Z^N_\beta}\sum_{\pi \in \pi^N} \bigg[\sum_{x \in\mathcal{T}^N} |x -\pi(x)| \bigg]\,e^{-\beta\sum_{x \in \mathcal{T}^N} |x -\pi(x)|} \\
\nonumber & = - \frac{1}{N^d} \frac{1}{Z^N_\beta} \frac{\partial}{\partial \beta} \bigg(  \sum_{\pi \in \pi^N} e^{-\beta\sum_{x \in \mathcal{T}^N} |x-\pi(x)|} \bigg) \\
 & = -  \frac{\partial}{\partial \beta} \Big( \frac{1}{N^d}  \log{Z^N_\beta} \Big) .
\end{align}
In the language of statistical mechanics, this is the usual relation between the mean energy and the temperature derivative of the free energy associated with a Boltzmann distribution.

\smallskip
We next observe that the partition function $Z^N_\beta$ can be written as the permanent of a certain symmetric matrix. Recall that the \textit{permanent} of an $n \times n$ matrix $A = (A_{i, j})$ is defined as
\[  \text{perm}(A) :=\sum_{\pi} \prod_{i=1}^nA_{i,\pi(i)}  ,\]
where $\pi$ runs over all permutations of the set $[\![1, N]\!]$. For each $N \ge 1$ and $\beta \ge 0$, one can write
\begin{equation}
\label{e:s2}
  Z^N_\beta =  \sum_{\pi \in \pi^N} e^{-\beta\sum_{x \in \mathcal{T}^N} |x-\pi(x)|}  =  \sum_{\pi \in \pi^N} \Pi_{x \in \mathcal{T}^N} e^{-\beta |x-\pi(x)|}  =   \text{perm}\{A^N_\beta\} 
  \end{equation}
where $A^N_\beta = ( (A^N_\beta)_{x,y} )_{x,y \in \mathcal{T}^N}$ is the $N^d \times N^d$ symmetric matrix with elements
\[  (A^N_\beta)_{x, y}  :=e^{-\beta |x-y|} . \]

\smallskip
The final step is to invoke an identity linking the permanent of a symmetric positive-definite matrix to certain moments of centred Gaussian vectors. Recall that to each $n \times n$ symmetric positive-definite matrix $A = (A_{i, j})_{1 \le i,j \le n}$ one can associate a centred Gaussian vector $(X_i)_{1 \le i \le n}$ with covariance matrix $A$, i.e., such that
\[ \mathbb{E}[X_i] = 0 \quad \text{and} \quad \mathbb{E}[X_i X_j] = A_{i,j} .    \]
The following formula is well-known \cite{F06, LW12, PV05}; to the best of our knowledge it first appeared in \cite{R62}, and notably was used in \cite{AdR98} to resolve the complex case of the \textit{polarisation constant conjecture}.

\begin{lemma}[Reed's formula]
\label{l:reed}
Let $A$ be an $n \times n$ symmetric positive-definite matrix. Then
\[   \text{perm}(A) = 2^{-n} \, \mathbb{E}\left[    \Pi_{i = 1}^n (X_i^2 + Y_i^2)   \right] ,     \]
where $X_i$ and $Y_i$ are independent copies of a centred Gaussian vector with covariance matrix~$A$.
\end{lemma}

Let us now combine the above observations and apply them to our setting. Observe first that Lemma \ref{l:reed} is applicable to $A^N_\beta$ since one can check that $A^N_\beta$ is positive-definite for each $N \ge 1$, $\beta > 0$ and dimension $d \ge 1$. The positive-definiteness of $A^N_\beta$ is a very special feature of the potential $V(x) = |x|$ that we consider, and ultimately derives from the fact that the \textit{Laplacian kernel} on $\mathbb{R}^d$ (also known as the \textit{Mat\'{e}rn kernel with shape parameter} $\nu = 1/2$)
\begin{equation}
\label{e:la}
\kappa(s, t) := e^{- |s-t| }
\end{equation}
is positive-definite in each dimension $d \ge 1$. Another notable Boltzmann SRP with this feature is the one given by the potential $V(x) = |x|^2$, arising naturally in the study of the ideal Bose gas (see Section \ref{s:lr} above). 

\smallskip
The centred Gaussian vector with covariance matrix $A^N_\beta$ has a natural description as the restriction of a continuous Gaussian field to a rescaled lattice box. Let $\Psi$ denote the stationary, almost surely continuous, centred Gaussian field on $\mathbb{R}^d$ with the Laplacian covariance kernel~\eqref{e:la}, i.e., such that
\[ \mathbb{E}[ \Psi(s) \Psi(t) ] = e^{-|s-t| } ; \]
such a Gaussian field exists by Kolmogorov's theorem. In one dimension, $\Psi$ is the classical \textit{Ornstein-Uhlenbeck process}, which as well as being Gaussian also enjoys the \textit{Markov property}; this fact is crucial for obtaining more detailed results in $d=1$ (see, however, the comments in Remark \ref{r:dmp} below on the existence of a \textit{(pseudo-)domain Markov property} in all odd dimensions). For each $N \ge 1$ and $\beta > 0$, define the rescaled lattice box
 \[ \mathcal{T}^N_\beta := \beta [\![0, N-1]\!]^d \subset \beta \mathbb{Z}^d . \]
Combining \eqref{e:s1}, \eqref{e:s2} and Lemma \ref{l:reed}, we obtain the following identities that underpin our proofs of  Theorems \ref{t:main1} and \ref{t:main2}:

\begin{proposition}
\label{p:id}
For each $N \ge 1$ and $\beta > 0$,
\[  Z^N_\beta = \text{perm}(A_\beta^N) = 2^{-N^d} \,  \mathbb{E}\left[    \Pi_{x \in \mathcal{T}^N_\beta} (X_x^2 + Y_x^2) \right]   \]
and
\[ D^N_\beta =  - \frac{\partial}{\partial \beta} \left(  \frac{1}{N^d} \log Z^N_\beta \right)  = - \frac{\partial}{\partial \beta} \left( \frac{1}{N^d} \log \mathbb{E}\left[    \Pi_{x \in \mathcal{T}^N_\beta} (X_x^2 + Y_x^2)   \right] \right)  , \]
where $X$ and $Y$ are independent copies of the centred Gaussian field $\Psi$. 
\end{proposition}

The remainder of the paper is essentially concerned with estimating the quantity
\begin{equation}
\label{e:e}
E^N_\beta := \mathbb{E}\left[    \Pi_{x \in \mathcal{T}^N_\beta} (X_x^2 + Y_x^2)   \right]   , 
\end{equation}
for which we use different techniques depending on the regime and the dimension. Let us state the main consequence, which is to establish asymptotics for the (log-)partition function:

 \begin{theorem}[Asymptotics for the partition function]
 \label{t:part}
For each $d \ge 1$, there exists a continuous, strictly decreasing, convex function $g^d: \mathbb{R}^+ \to \mathbb{R}$ satisfying $\lim_{c \to 0} g^d(c) = 0$ and $g^d(c) = - d \log c + O(1)$ as $c \to \infty$, such that, as $N \to \infty$,
 \[  \frac{1}{N^d} \log Z^N_\beta  = \begin{cases}
   \log(1/\beta^d) + O(1)  ,  &  \text{ if }  1/N \ll \beta \ll 1 , \\
    \frac{1}{N^d} \log ( ( N^d)! )  + g^d(c) + o(1)  ,  &  \text{ if }  \beta \sim c/N , \ c > 0 , \\
  \frac{1}{N^d} \log ( ( N^d)! ) + o(1)  , & \text{ if } \beta \ll 1/N . \\
  \end{cases}
    \]
In the case $d = 1$, the function $g^1$ is smooth, satisfies
\[ g^1(c) = -c/3 + O(c^2) \quad \text{as } c \to 0,\]
 and, as $N \to \infty$,
  \[ \frac{1}{N^d} \log Z^N_\beta  = \begin{cases}
  \log(1/\beta ) + \log 2 - 1 + O(\beta) + O(1/(\beta N)) ,  &  \text{ if }  1/N \ll \beta \ll 1, \\
 \frac{1}{N} \log (N!) + g^1(c) + o(1) ,  &   \text{ if }  \beta \sim c / N  , \ c > 0, \\
   \frac{1}{N} \log (N!)  -  \beta N/3  +  o(\beta N)    , & \text{ if } \beta \ll 1/N . \\
  \end{cases}
    \]
  The function $g^1$ has the representation
    \begin{equation}
    \label{e:g1}
     g^1(c) = 1 - h_c(0)^2 + \int_{s \in [0, c]}  \frac{2}{c} \log h_c(s) - \frac{1}{2} \left( h_c(s) +   \dot{h}_c(s) \right)^2  \, ds  
     \end{equation}
    where $h_c$ is the positive analytic function on $[0,c]$ that is the unique solution to the ODE in~\eqref{e:ode}.
 \end{theorem}
 
 \begin{remark}
 Since $Z^N_\beta = \text{perm}(A_\beta^N)$, Theorem \ref{t:part} also establishes the (log-)asymptotics of the permanent of the matrix $A^N_\beta$, which may be of independent interest. In one dimension such matrices are often called \textit{Kac-Murdoch-Szeg\H{o}} (KMS) matrices, and the asymptotics of their permanents have been considered elsewhere in the literature. For instance in \cite{LW12} it is proven that, for $d = 1$ and fixed $\beta > 0$,
 \[  \limsup_{N \to \infty}  \frac{1}{N} \log \rm{perm}(A^N_\beta) \le \log \Big( \frac{1+e^{-\beta}}{1-e^{-\beta}} \Big) .  \]
As $\beta \to 0$, 
 \[  \log \Big( \frac{1+e^{-\beta}}{1-e^{-\beta}} \Big) =   \log(2/\beta) + O(\beta)  ,\]
 and so, comparing to Theorem \ref{t:part}, this bound is weaker than ours in the regime $\beta \ll 1$.
 \end{remark}
 
 \begin{remark}
 \label{r:except}
The functions $f^d$ and $g^d$ in Theorems \ref{t:main1}, \ref{t:main2} and~\ref{t:part} are related by
\begin{equation}
\label{e:diff}
 f^d(c) = - \frac{d}{d c} g^d(c) 
 \end{equation}
at all points $c$ such that $g^d$ is differentiable. Since $g^d$ is convex it is differentiable except perhaps on a countable set $\mathcal{C}^d$ of discontinuities that are excluded from our results in Theorem~\ref{t:main1}. Since $g^1$ is smooth, we need no such restriction in one dimension.
\end{remark}

 \begin{remark}
 Our analysis of the critical regime $\beta \sim c/N$ in one dimension shows that the ($\log$-)asymptotics of the quantity
 \[ E^N_\beta = \mathbb{E}\left[    \Pi_{x \in \mathcal{T}^N_\beta} (X_x^2 + Y_x^2) \right]   \]
 are carried by the event in which $Y/X$ is roughly constant (see the proof of Proposition \ref{p:ode}), and indeed we are able to show that, as $N \to \infty$, 
 \[ \frac{1}{N} \log E^N_\beta = \frac{1}{N} \log  \mathbb{E}\left[    \Pi_{x \in \mathcal{T}^N_\beta} X_x^2 \right]  + o(1)  .\]
 This latter expression is what leads us to the representation for $g^1$ in \eqref{e:g1} in terms of a \textit{one-dimensional} variational problem, rather than a two-dimensional variational problem as one might expect. 
 \end{remark}
 
 \begin{remark}
 \label{r:dmp}
 The representation for $g^1$ in terms of the solution to the ODE in \eqref{e:ode} is ultimately due to the Markov property of $\Psi$ in one dimension (see the analysis in Section \ref{s:ld}). Since, in fact, $\Psi$ satisfies a certain \textit{(pseudo-)domain Markov property} in all odd dimensions (see, e.g., \cite[Theorem A.3]{adler}), it is possible that a variant of our analysis might extend to these cases as well. 
\end{remark}

\begin{remark}
 A comparison with the results of \cite{M16a} (see, e.g., Proposition~\ref{p:vf}) would allow us to deduce that $g^1$ is actually \textit{strictly} convex. Since we are unable to prove this property directly from our approach, we prefer to omit it from the statement of our results.
 \end{remark}

\begin{figure}[ht]
\begin{tikzpicture}[scale=1.1]	
\draw[->, thick] (0,3)--(3,3) node[right]{$c$};
\draw[->, thick] (0,3)--(0,1) node[below]{$g^1(c)$};
 \draw[scale=1.1,domain=0:2.8,smooth,variable=\x] plot ({\x},{2/(1+\x))+0.7});
\end{tikzpicture}
\caption{A rough illustration of the function $g^1$ in Theorem \ref{t:part}.}
\label{fig:g}
\end{figure}
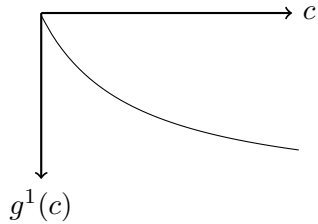
 
Given Theorem \ref{t:part}, to complete the proofs of Theorem \ref{t:main1} and \ref{t:main2} it remains to extract bounds for $D^N_\beta =  - \frac{\partial}{\partial \beta} (  \frac{1}{N^d} \log Z^N_\beta )$. This is possible via the following convexity statement:

\begin{lemma}
\label{l:convex}
For each $N \ge 2$, the function
\[ \beta \mapsto   \frac{1}{N^d} \log Z^N_\beta \]
is strictly decreasing and convex on $(0, \infty)$. 
\end{lemma}  
\begin{proof}
It is clear that $Z^N_\beta = \text{perm}(A^N_\beta)$ is strictly decreasing as a function of $\beta$ since each off-diagonal element of $A^N_\beta$ is positive and strictly decreasing in $\beta$. Moreover, since the partition function of a Boltzmann distribution is always log-convex with respect to the inverse temperature, $Z^N_\beta$ is log-convex as a function of $\beta$. 
\end{proof}

\subsection{The variational formula}
\label{s:vf}

In this section we discuss the variational formula that is used to describe the functions $g^d$ in Theorem \ref{t:part}, and hence also the function $f^d$ in Theorems~\ref{t:main1} and \ref{t:main2}. As shown in Section \ref{s:ld}, this description arises naturally out of the large deviation theory of Gaussian fields, via the connection between SRPs and Gaussian fields explained in Section~\ref{s:h}. \vspace{0cm}

Large deviation principles for Gaussian fields involve the notions of \textit{entropy} and \textit{energy} for Gaussian measures, which we recall now. For a domain $D \subset \mathbb{R}^d$, let $C(D)$ denote the space of continuous functions $f: D \to \mathbb{R}$. By the standard theory of Gaussian fields, to each stationary, almost surely continuous, centred Gaussian field $\Psi$ on $\mathbb{R}^d$ with covariance kernel~$\kappa$ one can associate a \textit{reproducing kernel Hilbert space} (RKHS) (also sometimes called the \textit{Cameron--Martin space}) $\mathcal{H} \subset C(\mathbb{R}^d)$; this is formed by completing the space of \textit{finite linear combinations} of the covariance kernel 
\[  f( \cdot ) =  \sum_{1 \le i \le n} a_i \kappa(s_i,  \cdot)   \ , \quad a_i \in \mathbb{R}, s_i \in \mathbb{R}^d , \]
 equipped with the inner product  
\begin{equation}
\label{e:rkhsip}
  \bigg\langle \sum_{1 \le i \le n} a_i \kappa(s_i, \cdot), \sum_{1 \le j \le n} a'_j \kappa(s'_j, \cdot) \, \bigg\rangle_{\mathcal{H}}  = \sum_{1 \le i, j \le n }  a_i a_j' \kappa(s_i, s_j')   ,
  \end{equation}
which satisfies in particular the \textit{reproducing property}
\begin{equation}
\label{e:repro}
    f(x) =\left\langle f(\cdot), \kappa(x, \cdot) \right\rangle_{\mathcal{H}} . 
    \end{equation}
See \cite{bertho04, jansen} for background on the RKHS of Gaussian fields.

\smallskip
For each $c > 0$ one can similarly define the RKHS $\mathcal{H}_c$ of the restricted field $\Psi|_{[0, c]^d}$ by restricting the functions in $\mathcal{H}$ to the domain $[0, c]^d$. The norms in $\mathcal{H}$ and $\mathcal{H}_c$ are related by a restriction property; namely for $\bar f \in \mathcal{H}_c$,
\begin{equation}
\label{e:restrict}
  \| \bar f \|_{\mathcal{H}_c} = \inf_{f \in \mathcal{H} : f|_{[0, c]^d} = \bar f} \| f \|_\mathcal{H}  
  \end{equation}
with the infimum attained by some $f \in \mathcal{H}$.

\smallskip
Recall that $\Psi$ denotes the stationary, almost surely continuous, centred Gaussian field on~$\mathbb{R}^d$ with covariance
\[ \kappa(s, t) =  \mathbb{E}[ \Psi(s) \Psi(t) ] = e^{-|s-t| } . \]
Henceforth, let $\mathcal{H}$ denote the RKHS of the field $\Psi$, and for each $c > 0$, let $\mathcal{H}_c$ denote the RKHS of $\Psi|_{[0, c]^d}$; the respective norms in these spaces are denoted $\| \cdot \|_\mathcal{H}$ and $\| \cdot \|_{\mathcal{H}_c}$.

\smallskip
In the case $d = 1$, the Gaussian process $\Psi$ is the \textit{Ornstein-Uhlenbeck process}, and~$\mathcal{H}$ is the classical \textit{Cameron-Martin space} consisting of functions in $L^2(\mathbb{R})$ that are absolutely continuous and whose (weak) derivative is also in $L^2(\mathbb{R})$; the norm in $\mathcal{H}$ is given by
\begin{align}
\label{e:oun}
 \|f\|^2_\mathcal{H} &=   \frac{1}{2} (\| f \|_2^2 + \|f'\|_2^2 ) = \frac{1}{2}  \int_{s \in \mathbb{R}}  f(s)^2 + f'(s)^2    \, ds   =  \frac{1}{2} \int_{s \in \mathbb{R}}  \left( f(s) + f'(s) \right)^2   \, ds ,
 \end{align}
 where the latter equality is via integration by parts ($f, f' \in L^2$ implies that $f$ decays at infinity); see also Section~\ref{s:ld} where we show how \eqref{e:oun} may be derived. Similarly, the norm in~$\mathcal{H}_c$ is given by
\begin{align}
\label{e:oun2}
 \|f\|^2_{\mathcal{H}_c} & =  \frac{1}{2} \left( f(0)^2 + f(c)^2 \right) +  \frac{1}{2}  \int_{s \in [0,c]}  f(s)^2 + f'(s)^2    \, ds \\
 \nonumber &   =  f(0)^2 + \frac{1}{2} \int_{s \in [0,c]}  \left( f(s) + f'(s) \right)^2   \, ds .
 \end{align}
 Explicit descriptions of $\|\cdot\|_\mathcal{H}$ and $\|\cdot\|_{\mathcal{H}_c}$ in higher dimensions are not as simple, although, as we explain in Section \ref{s:ld}, in odd dimensions they can also be represented as integrals over the function $f$ and its derivatives; this `local' expression for $\|\cdot\|_\mathcal{H}$ is related to the (pseudo-)domain Markov property enjoyed by $\Psi$ in odd dimensions (see Remark \ref{r:dmp} above).

\smallskip
Define the \textit{entropy}, or \textit{large deviation rate}, of a function $f \in C(\mathbb{R}^d)$ to be
\[   I[f]  := \begin{cases}
\frac{1}{2} \|f\|^2_\mathcal{H} , & \text{if } f \in \mathcal{H} , \\
\infty, & \text{else}.
\end{cases}   \]
For each $c > 0$ and each pair of functions $f_1, f_2 \in C(\mathbb{R}^d)$ define the \textit{energy} of the pair to be 
\[  J_c[f_1, f_2] :=  c^{-d} \int_{s \in [0, c]^d} \log (   f_1(s)^2 + f_2(s)^2    )  \, ds .   \]
The variational problem in Gaussian space that defines $g^d$ arises out of a balance between energy and entropy:
\begin{equation}
\label{}
 (VP\!:\!G) \qquad \text{maximise}  \quad J_c[f_1, f_2]    -  I[f_1]  - I[f_2] \quad \text{over} \quad f_1, f_2 \in C(\mathbb{R}^d) . 
 \end{equation}
 
\smallskip
We analyse the variational problem $(VP\!:\!G)$ in Section \ref{s:ld}, where we prove in particular the following properties:

\begin{proposition}
\label{p:vprop}
For each $d \ge 1 $ and $c > 0$, the variational problem $(VP\!:\!G)$ has a finite solution $V_{G}^{d} (c)$ that is attained for a pair $f_1, f_2 \in C(\mathbb{R}^d)$. The function $c \mapsto V_G^d(c)$ is continuous, strictly decreasing, convex, and satisfies 
\[  \lim_{c \to 0} V_{G}^{d}(c) = \log 2 - 1  \quad  \text{and} \quad   V_G^d(c) = - d \log c + O(1)  \ \quad \text{as } c \to \infty.   \]
Moreover, the function $c \mapsto V_G^1(c)$ is smooth, satisfies
 \[   V_G^1(c)   = \log 2 - 1  -c/3 + O(c^2) \quad \text{as } c \to 0 ,\]
 and has the representation
\[  V_G^1(c) = \log 2   - h_c(0)^2 +  \int_{s \in [0, c]} \frac{2}{c}  \log h_c(s) - \frac{1}{2} \left( h_c(s) + \dot{h}_c(s)  \right)^2  \, ds  \]
where $h_c$ is the positive analytic function on $[0, c]$ that is the unique solution to the ODE in~\eqref{e:ode}.
\end{proposition}

Recall the Gaussian moment $E^N_\beta$ defined in \eqref{e:e}. In Section \ref{s:ld} we use the classical theory of large deviations for Gaussian measures to prove that, if $\beta \sim c/N$,
\[ \frac{1}{N^d} \log E^N_\beta =    d \log N + V_G^d(c)   + o(1) .\]
Along with Proposition \ref{p:id} and Stirling's formula
\[     \frac{1}{N^d} \log (  (N^d)! )  = d \log N - 1  + o(1) , \]
 this implies that
\begin{equation}
\label{e:permlimit}
   \frac{1}{N^d} \log Z^N_\beta =     \frac{1}{N^d} \text{perm}(A^N_\beta) =  \frac{1}{N^d} \log (  (N^d)! )  + V_G^d(c) - \log 2   + 1 +o(1) ,
 \end{equation}
which completes the proof of Theorem \ref{t:part} in the critical regime for $g^d$ defined as
\[  g^d(c) = V_G^d(c) - \log 2 + 1 . \]
Moreover, in Section \ref{s:final} we show that \eqref{e:permlimit} implies also that, as $N \to \infty$,
 \[  D^N_\beta =   N \times - \frac{d}{dc} V_G^1(c) \times ( 1 + o(1) ) ,\]
 which also completes the proof of Theorem \ref{t:main2} in the critical regime.
  
\smallskip
In light of the the results in \cite{M16a} discussed in Section \ref{s:lr} above, an immediate corollary of Theorem~\ref{t:main2} is the equivalence between the solutions of the two ostensibly distinct variational formulae $(VP\!:\!P)$ and $(VP\!:\!G)$.

\begin{proposition}[Equivalence of variational formulae]
\label{p:vf}
For each $c > 0$,
\[ V_P(c) =  V_G^1(c) - \log 2  + 1.\]
\end{proposition}

We do not have a direct proof of the equality in Proposition \ref{p:vf}, which in isolation appears quite mysterious. As we mention in Section \ref{s:lr} above, the approach of \cite{M16a} can likely be extended to higher dimensions, which would give an analogous relationship between $V_G^d$ and the natural higher dimensional analogues of $V_P$.

\subsection{Open questions and discussion}

Our results and methods raise many natural questions, which we discuss briefly now:
\begin{enumerate}
\item Can our results be extended to general potential functions $V(x)$? As discussed in Section \ref{s:h}, the connection to Gaussian fields requires the matrix~$A^N_\beta = (e^{-\beta V(x-y)})_{x,y \in \mathcal{T}^N}$ to be positive-definite, but this is still true for many other natural potentials, in particular for the potential $V(x) = |x|^2$ that relates to the ideal Bose gas. One complication is that, in assessing the band-structure of these generalised models, there is an additional natural parameter to include: the scale of the lattice. In other words, one should consider the generalised model
\[   P(\pi) \propto e^{-\beta \sum_{x \mathcal{T}^N}V( (x - \pi(x)) / L  ) } , \]
for which the band-width will depend on all of the parameters $\beta, N$ and $L$. Only in the special case of potential $V(x) = |x|$ are the parameters $\beta$ and $L$ directly related via linear rescaling.
\item Can we say anything about the model in which the potential function $V(x) = |x|$ is replaced by a periodic version, i.e., $V(x) = \| x \|_N$, where $\| \cdot \|_N$ denotes the Euclidean distance on the lattice-box $\mathcal{T}^N$ with periodic boundary conditions? In this model one loses the connection to Gaussian fields since the analogue of $A^N_\beta$ is no longer positive-definite. However, various other techniques become applicable, for instance one can apply the \textit{van der Waerden inequality} to give an immediate lower bound on $Z^N_\beta = \text{perm}(A^N_\beta)$ which is fairly sharp.
\item A important statistical feature of random permutations is their cycle structure, especially the lengths of the longest cycles. For our model, a question of major interest is to determine the scale $\beta$, depending on $N$, on which macroscopic cycles emerge as $N \to \infty$, i.e.\ such that there are cycles of length comparable to $N$ with non-negligible probability. Based on considerations of universality, the critical band-width at which macroscopic cycles emerge should depend on the dimension as
\[ \beta = \begin{cases}
\Theta(\sqrt{N}), & d = 1, \\
\Theta(\sqrt{\log N}), & d = 2, \\
\Theta(1), & d \ge 3 .
\end{cases} \] 
This is known in the model of the ideal Bose gas in \eqref{e:ann} \cite{BU09, EP17}, and has recently been confirmed for $d =1 $ also in the Mallows model \cite{GP16}. It is thought that these scales are also related to the transition between delocalised/localised eigenvectors in random band matrices when the band-width attains a critical size, originally conjectured in \cite{FM91} (see also the discussion in \cite{GP16}). In light of our results in Theorems \ref{t:main1} and \ref{t:main2}, this leads us to the following conjecture.

\begin{conjecture}[Emergence of macroscopic cycles]
\label{c:emerge}
Let
\[ C_\beta^N = \frac{1}{N^d} \sum_{x \in \mathcal{T}^N} E[\text{length of the cycle containing } x] \]
denote the mean cycle length per site. Then there exists a $\beta_c > 0$, depending on $d$, such that as $N \to \infty$,
\[ C_\beta^N =  \begin{cases} 
  \min\{ \Theta( 1/\beta^2) , \Theta(N) \},  & d =1, \\
    \min\{ e^ { \Theta(1/\beta^2) }, \Theta(N) \}, & d = 2 , \\ 
       \Theta( N )  ,  & d \ge 3 , \, \beta < \beta_c, \\
         \Theta(1) ,  & d \ge 3 , \, \beta > \beta_c .
  \end{cases} \]
  \end{conjecture}
 \end{enumerate}

\subsection{Overview of the remainder of the paper}

The rest of the paper is organised as follows. In Section \ref{s:ou} we study the subcritical regime $\beta \gg 1/N$ of the one-dimensional model, and in particular show how to exploit the Markov property of the Ornstein-Uhlenbeck process to derive the rate-of-convergence in Theorem \ref{t:main2}. In Section \ref{s:ld} we study the critical regime $\beta \sim c/N$ in all dimensions by applying large deviation theory for Gaussian fields. Here we also study the variational problem $(VP\!:\!G)$ that was introduced in Section \ref{s:vf} above. In Section \ref{s:subhigh} we use general Gaussian estimates to study the subcritical regime of the model in all dimensions $d \ge 1$. Finally, in Section \ref{s:final} we combine the analysis to complete the proofs of Theorems \ref{t:main1}, \ref{t:main2} and \ref{t:part}.


\medskip

\section{The model in one dimension: Kernel expansion}
\label{s:ou}

In this section we fix $d=1$ and give an estimate on the quantity
\[ E^N_\beta  =   \mathbb{E}\left[    \Pi_{x \in \mathcal{T}^N_\beta} (X_x^2 + Y_x^2)   \right]  = 2^{N} Z^N_\beta   , \]
where $\mathcal{T}^N_\beta = \beta [\![0, N-1]\!] \subset \beta \mathbb{Z}  $ and  $X_x$ and $Y_x$ are independent Ornstein-Uhlenbeck processes (see Section \ref{s:h}). Applied to the subcritical regime $\beta \gg 1/N$, this will allow us to prove the following:

\begin{proposition}
\label{p:noncrit}
Let $1/N \ll \beta \ll 1$. Then, as $N \to \infty$,
\[ \frac{1}{N} \log E^N_\beta =   \log( 4/\beta ) - 1 + O(\beta) + O(1/(\beta N)) . \]
\end{proposition}

The key feature of the one-dimensional model is the fact that the Ornstein-Uhlenbeck process is Markovian. We exploit this fact by developing a kernel expansion for $E^N_\beta$ that is valid for all $\beta > 0$; this analysis is somewhat reminiscent of methods used in \cite{FM91, FM94, LMOS83, Kac66}.

\subsection{The kernel expansion}
For each $t > 0$, let $e^{t \Delta}$ denote the \textit{heat semi-group}, i.e.\ the operator on $f \in L^2(\mathbb{R}^2 ) $ that acts via 
\[    (e^{t \Delta} f )(s_1) = \int_{s_2 \in \mathbb{R}^2  }   p_t(s_1, s_2)  f(s_2) \,   ds_2,\]
where 
\[ p_t(x, y) :=   \frac{1}{4 \pi t} e^{-|x-y|^2/(4 t)  }  \]
is the \textit{heat kernel} (i.e.\ the Gaussian kernel with variance $\sigma^2 = 2t$).

\begin{proposition}[Kernel expansion]
\label{p:int}
For $N \ge 2$ and $\beta > 0$ the quantity $E^N_\beta$ has the kernel expansion
\[ E^N_\beta =   2^{N} \times \frac{e}{\pi} \times  \frac{e^{-\beta}(1+ e^{-\beta})}{1-e^{-\beta}} \times  \left( \frac{e^\beta (1+ e^{-\beta} )}{e(1-e^{-\beta})} \right)^{N} \times \int_{s \in \mathbb{R}^2 } v(s) ( \mathcal{K}^{N-1} v)(s)  \, ds, \]
where 
\[  v(s) := |s| \exp \left\{ - \frac{1}{2} \cdot \frac{1 + e^{-\beta}}{1-e^{-\beta}} |s|^2 \right\}  \]
 and the operator
\[  \mathcal{K} := f  e^{t \Delta}  f \]
on $L^2(\mathbb{R}^2)$ is a weighted heat semi-group with weight function and `time' given respectively by 
\[ f(s) := |s| e^ {  \frac{1}{2}  ( 1 -  |s|^2 ) }  \quad \text{and} \quad t :=  \frac{(1-e^{-\beta})^2}{4 e^{-\beta}} . \]
\end{proposition}

\begin{remark}
Note that although both $v$ and $t$ depend on $\beta$, to ease notation we have dropped the explicit dependence.
\end{remark}

\begin{proof}
The Markov property of the Ornstein-Uhlenbeck process implies that, for any $s, t \ge 0$,
\[   X_{s + t}  |   (X_u)_{u \le s}   \stackrel{d}{=}  \mathcal{N} \left( e^{-t} X_s  , 1 - e^{-2t } \right) ,\]
and so successive conditioning on $(X_{\beta i },Y_{\beta i }) = (x_i, y_i)$ yields a representation of $E^N_\beta$ as the integral of $\Pi_{i=0}^{N-1} (x_i^2 + y_i^2)$ against the Gaussian density
\begin{align*}
& (2 \pi)^{-N} \cdot (1 - e^{-2\beta})^{-(N-1)} \cdot  \\
& \quad  \cdot \exp \Big\{ - \frac{1}{2}(x_0^2 + y_0^2)  - \sum_{i = 1}^{N-1}   \frac{1}{2(1 - e^{-2\beta}) } \left( (x_{i } - e^{-\beta} x_{ i-1 } )^2 + (y_{i } - e^{-\beta} y_{ i-1 } )^2 \right) \Big\} . 
\end{align*}
Abbreviating $s_i = ( x_{i}, y_{i}   )$ and rearranging terms, the above density is equal to
\begin{align*}
& (2 \pi)^{-N} \cdot (1 - e^{-2\beta})^{-(N-1)} \cdot  \\
& \quad \cdot \exp \Big\{ - \frac{1}{4} \Big( |s_0|^2 + |s_{N-1}|^2   \Big)  -  \sum_{i =0}^{N-2}  \Big(   \frac{1 - e^{-\beta}}{4(1+e^{-\beta})} ( |s_i|^2 + |s_{i +1}|^2 )  + \frac{e^{-\beta}}{2(1 - e^{-2\beta})}  |s_i - s_{i+1} |^2    \Big) \Big\} . 
\end{align*}
Applying the change of variables
\[  s_i \mapsto  \sqrt{ \frac{2 (1+e^{-\beta} )}{ 1 - e^{-\beta} } } s_i    \]
yields that $E^N_\beta$ is the integral of 
\begin{align*}
&  \Pi_{i=0}^{N-1} |s_i|^2  \cdot (2 \pi)^{-N} \cdot (1 - e^{-2\beta})^{-(N-1)} \cdot \left( \frac{2 ( 1 + e^{-\beta} ) }{1 - e^{-\beta} } \right)^{2N} \cdot \\
& \quad \cdot \exp \Bigg\{ - \frac{1}{2} \Bigg(   \frac{1+ e^{-\beta}}{1 - e^{-\beta} } \Big(   |s_0|^2 + |s_{N-1}|^2 \Big)     +  \sum_{i =0}^{N-2}  \Big(   \Big( |s_i|^2 + |s_{i +1}|^2 \Big)  + \frac{2 e^{-\beta}}{ (1 - e^{-\beta})^2 }  |s_i - s_{i+1} |^2    \Big) \Bigg)  \Bigg\} \\
&  = (2 \pi)^{-N} \cdot (1 - e^{-2\beta})^{-(N-1)} \cdot \left( \frac{2 ( 1 + e^{-\beta} ) }{1 - e^{-\beta} } \right)^{2N} \cdot  e^{-(N-1)} \cdot \left( \frac{\pi(1 - e^{-\beta})^2 }{e^{-\beta} } \right)^{N-1} \cdot \\
& \quad \cdot  v(s_0) v(s_{N-1}) \prod_{i =0}^{N-2}  f(s_i) p_t(s_i, s_{i+1}) f(s_{i+1}) , 
\end{align*}
which is equivalent to the claimed expression.
\end{proof}

In light of Proposition \ref{p:int}, in order to prove Proposition \ref{p:noncrit} it remains to give bounds on the (iterated) integral
\begin{equation}
\label{e:innerproduct}
 \int_{s \in \mathbb{R}^2}  v(s)  (\mathcal{K}^{N-1} v)(s) \, ds .
 \end{equation}
 In particular, we prove the following:
 
 \begin{proposition}
 \label{p:boundint}
If $\beta \ll 1$ then,
\[    \Big|   \log  \int_{s \in \mathbb{R}^2}  v(s) (\mathcal{K}^{N-1} v)(s) \, ds  \Big|  = O(\beta N) +   O(1/\beta)   . \]
\end{proposition}

We prove Proposition~\ref{p:boundint} by applying a spectral analysis of the operator $\mathcal{K}$. Before turning to the proof, let us show how it implies Proposition \ref{p:noncrit}:

\begin{proof}[Proof of Proposition \ref{p:noncrit}]
Suppose $1/N \ll \beta \ll 1$. Combining Proposition \ref{p:int} and the first statement of Proposition \ref{p:boundint} yields
\[  \frac{1}{N} \log E^N_\beta  = \log2  + \log \left(  \frac{e^\beta (1 + e^{-\beta}) }{e(1-e^{-\beta})} \right)  +  O((- \log \beta) / N) + O(1/N)  + O(\beta ) + O(1/(\beta N)) . \]
Since, as $\beta \to 0$,
\[  \log \left(  \frac{e^\beta (1 + e^{-\beta}) }{e(1-e^{-\beta})} \right)  = \log(2 / \beta) - 1 + O(\beta ),  \]
and 
\[ ( - \log \beta) / N  \ll O(1/ (\beta N) ) , \]
we have the result.
\end{proof}

\begin{remark}
Although the kernel expansion in Proposition \ref{p:int} is also valid in the supercritical regime $\beta \ll 1/N$, it is more difficult to control in that regime; in fact, what differentiates the subcritical and supercritical regimes is that, in the former, the main contribution to $E^N_\beta$ is carried by the principle eigenvalue of $\mathcal{K}$.
\end{remark}

\subsection{Spectral analysis of the kernel}
f
We now prove Proposition~\ref{p:boundint} by applying a spectral analysis of the operator $\mathcal{K}$. We begin with some simple facts about this operator; here the fact that the weight function $f$ has maximum value $1$ attained on the annulus $\{|s| = 1\}$ will be crucial.

\begin{lemma}[Spectral analysis of the operator $\mathcal{K}$]
\label{l:prop}
~\begin{enumerate} 

\item There exists an orthonormal basis of eigenvectors $(\varphi_i)_{i \in \mathbb{N}}$ of $\mathcal{K}$ with associated eigenvalues 
\[ 1 >  \lambda_1 > \lambda_2 \ge \ldots  > 0  \] 
satisfying
\[ \| \lambda_i \|_2 \le  \frac{1}{\sqrt{8 \pi t}}  \, \| f \|_2   ,   \]
where $ \| \lambda_i \|_2$ denotes the $\ell^2$-norm of the vector $(\lambda_i)_{i \in \mathbb{N}}$.
\item As $\beta \to 0$,
\[ \lambda_1 = 1 -  O(\beta)  . \]

\item  Each eigenfunction satisfies 
\[ \| \varphi_i \|_1 \ge \sqrt{8 \pi t}  \,  \lambda_i  ,\]
and, for every Borel set $A \subseteq \mathbb{R}^2$ such that $\| f \id_{A^c} \|_\infty  \le \lambda_i$,
\[  \| \varphi_i \id_A \|_1 \ge  \sqrt{8 \pi t}  \, \lambda_i  \times \frac{\lambda_i - \| f \id_{A^c} \|_\infty }{1 - \|f \id_{A^c} \|_\infty }.   \]
\end{enumerate}
\end{lemma}

\begin{proof}
\textbf{(1)} Remark that $\mathcal{K}$ is an integral operator with kernel 
\[ K(s_1, s_2) = f(s_1) p_t(s_1, s_2) f(s_2)  , \]
and note that
\begin{align*}
 \| K \|^2_2  &  =  \int  f(s_1)^2    \int  p_t(s_1, s_2)^2   f(s_2)^2 \,  ds_2 ds_1   \\
  & =\frac{1}{8 \pi t}  \int  f(s_1)^2    \int  p_{t/2} (s_1,s_2)  f(s_2)^2 \,  ds_2 ds_1   \\
  &  \le \frac{ \|f\|_\infty^2 }{8 \pi t}  \int \! \!   \int  p_{t/2} (s_1,s_2)  f(s_2)^2 \,  ds_2 ds_1 \\
   & = \frac{1}{8 \pi t}  \|  f \|_2^2  ,
  \end{align*}
where in the second equality we used an explicit computation with the heat kernel, and in the last step we used the fact that $\|f\|_\infty = 1$ and the fact that the heat kernel preserves the $L^1$-norm of $f^2 \ge 0$. Hence $\mathcal{K}$ is in fact a Hilbert-Schmidt operator, with Hilbert-Schmidt norm
  \[   \| \mathcal{K}\|_{\text{HS}} = \| \lambda_i \|_2   = \| K \|_2 \le  \frac{1}{\sqrt{8 \pi t}}  \|  f \|_2.    \]
Moreover,~$\mathcal{K}$ is positive-definite since $K$ is a weighted heat kernel with positive weights, and so its eigenvalues are non-negative. Finally, since $K$ is (point-wise) positive, by the Perron-Frobenius theorem the principal eigenvalue $\lambda_1$ is simple and $\varphi_1 \ge 0$. Moreover $\lambda_1 < 1$ since
  \[  \lambda_1 \| \varphi_1 \|_1 = \| f e^{t \Delta} f \varphi_1 \|_1 < \| f \|_{\infty}^2 \| e^{t \Delta} \varphi_1 \|_1 = \| f\|_{\infty}^2 \| \varphi_1 \|_1 = \| \varphi_1 \|_1 ,\]
where the strict inequality in the second step comes from the fact that $f = \|f\|_\infty$ only on the annulus $\{|s| = 1\}$ (on which $e^{t \Delta} \varphi_1$ cannot be exclusively supported), and the third step holds since the heat kernel preserves the $L^1$-norm of $\varphi_1 \ge 0$.

\smallskip
\textbf{(2)} We use the min-max formula for the principal eigenvalue evaluated against a well-chosen test function. For each $c > 0$ define the ball $B_c := \{s : |s - 1| \le c \}$, and let $\varphi_{1;c}$ and $-\lambda_{1;c}$ denote respectively the principal Dirichlet ($L^2$-normalised) eigenfunction and eigenvalue of $\Delta$ in $B_c$, noting that $\varphi_{1;c} \ge 0$. Applying the min-max formula to $\varphi_{1;c}$, 
\[    \lambda_1 \ge   \| f e^{t \Delta} f \varphi_{1;c} \|_{2}  \ge  \big(   \inf_{x \in B_c} f^2(x)   \big)   \| e^{t \Delta} \varphi_{1;c} \|_{2}  =  \big( \inf_{x \in B_c} f^2(x)  \big) e^{- t \lambda_{1;c}}  . \]
Now observe that $f|_{B_c} > 1 - O(c^2)$, $t = O(\beta^2)$, and by the rescaling property of Dirichlet eigenvalues,
\[  \lambda_{1;c} = O(c^{-2}) \quad \text{as } c \to 0   .  \]
Specialising to $c = \beta^{1/2} \gg \sqrt{t}$, as $\beta \to 0$,
\[     \lambda_1 \ge  (1 - O(c^2) )  (1 - O(t/c^2) )   = 1 - O(\beta) .\]

\smallskip
\textbf{(3)} We lower bound $\|\varphi_i\|_1$ via
\begin{align*}
   \lambda_i   = \lambda_i  \| \varphi_i \|_2  = \| f e^{t \Delta} f \varphi_i\|_2     \le \| f\|_\infty \| p_t(\cdot, 0) \|_2 \|f \|_\infty \| \varphi_i \|_1  = \frac{1}{\sqrt{8 \pi t}} \|\varphi_i\|_1    ,
  \end{align*}
  where the inequality in the third step is Young's convolution inequality, and the fourth step is via an explicit computation with the heat kernel. Moreover,
  \begin{align*}
    \lambda_i \| \varphi_i \|_1 \le \|    f e^{t \Delta} f |\varphi_i|  \|_1 & = \|    f e^{t \Delta} f |\varphi_i| \id_{A^c} \|_1 + \|    f e^{t \Delta} f |\varphi_i| \id_{A} \|_1  \\
  & \le \| f \|_\infty \, \|f \id_{A^c} \|_\infty \, \| \varphi_i \id_{A^c} \|_1 + \| f \|_\infty^2  \, \|\varphi_i \id_{A} \|_1 \\
  & = \|f \id_{A^c} \|_\infty \| \varphi_i  \|_1 + (1 - \|f \id_{A^c} \|_\infty ) \|  \varphi_i \id_{A} \|_1 .
    \end{align*}
 Rearranging, this gives
 \[  \| \varphi_i \id_A \|_1 \ge  \|  \varphi_i \|_1  \times \frac{\lambda_i - \| f \id_{A^c} \|_\infty }{1 - \|f \id_{A^c} \|_\infty }   \]
 which, after combining with the lower bound on $ \|\varphi_i\|_1$, yields the result. 
\end{proof}

We can now complete the proof of Proposition \ref{p:boundint}:

\begin{proof}[Proof of Proposition \ref{p:boundint}]
Consider the spectral expansion
\[     \int_{s \in \mathbb{R}^2}  v(s) (\mathcal{K}^{N-1} v)(s) \, ds  = \sum_{i \ge 1} \lambda_i^{N-1} \langle \varphi_i, v \rangle^2 . \]
Since $\lambda_i \in (0, 1)$ by property $(1)$ of Lemma \ref{l:prop}, this implies the bounds, for each $N \ge 3$,
\begin{equation}
\label{e:boundint1}
  \lambda_1^{N-1}  \langle \varphi_1, v \rangle^2  \le     \int_{s \in \mathbb{R}^2}  v(s) (\mathcal{K}^{N-1} v)(s) \, ds   \le  \| \lambda_i \|_2^2  \, \sup_i   \langle \varphi_i, v \rangle^2 . 
  \end{equation}
Let us treat first the upper bound in \eqref{e:boundint1}. By the Cauchy-Schwartz inequality and the upper bound on $\|\lambda_i\|_2$ in property $(1)$ of Lemma~\ref{l:prop} (recall also that $\|\varphi_i\|_2 = 1$),
\[   \| \lambda_i \|_2^2  \, \sup_i   \langle \varphi_i, v \rangle^2 \le  \| \lambda_i \|_2^2   \, \|  v \|_2^2 \le \frac{1}{8 \pi t} \| f \|^2_2   \|  v \|_2^2  =  \frac{ e^{-\beta}}{2 \pi (1-e^{-\beta})^2}   \| f \|^2_2  \| v \|_2^2 .  \]
Observe next that, as $\beta \to 0$,
\[   \frac{ e^{-\beta}}{2 \pi (1-e^{-\beta})^2}  = O(\beta^{-2})  , \]
and also that $\|f\|_2 < \infty$ and $\|v\|_2 \to 0$ as $\beta \to 0$. Hence
\[     \| \lambda_i \|_2^2  \, \sup_i   \langle \varphi_i, v \rangle^2   = o(\beta^{-2}) ,\]
We turn to the lower bound in \eqref{e:boundint1}. Appealing to property (3) in Lemma~\ref{l:prop} (and recalling that $\varphi_1$ and $v$ are non-negative),
\begin{align*}
    \lambda_1^{N-1}  \langle \varphi_1, v \rangle^2 &  \ge \lambda_1^{N-1}  \sup_{A \subseteq \mathbb{R}^2} \left( \inf_{s \in A} v(s) \| \varphi_1 \id_A \|_1 \right)^2      \\
    & \ge 2 \pi e^\beta ( 1 - e^{-\beta} )^2 \lambda_1^{N+1} \, \sup_{A \subseteq \mathbb{R}^2}   \,    \left( \frac{ \inf_{s \in A} v(s)   (\lambda_1 -  \|f \id_{A^c} \|_\infty ) }{1 - \|f \id_{A^c} \|_\infty  }  \right)^2 .
    \end{align*}
Now define the annulus $A = \{ s : |s| \in (1/2, 3/2) \}$, which satisfies, as $\beta \to 0$,
\[      \inf_{s \in A} v(s)  \ge  e^{ - O(1/\beta) }  \]
and $\| f \id_{A^c} \|_\infty \in (0, 1)$. Since $\lambda_1 = 1 - O(\beta)$ by property $(2)$ of Lemma~\ref{l:prop}, as $\beta \to 0$,
\[   \lambda_1^{N-1}  \langle \varphi_1, v \rangle^2 \ge  (1 + O(\beta))^{N}  e^{-O(1/\beta)}.  \] 
To sum up, we have shown that, as $\beta \to 0$,
\[   (1 + O(\beta))^{N}  e^{-O(1/\beta)} \le    \int_{s \in \mathbb{R}^2}  v(s) (\mathcal{K}^{N-1} v)(s) \, ds  \le  o(\beta^{-2})  \]
which, upon taking the logarithm, establishes the claim.
\end{proof}


\medskip

\section{The critical regime: Large deviations and variational problems}
\label{s:ld}

In this section we study the Gaussian expectation $E^N_\beta$ in the critical regime $\beta \sim c / N$ in all dimensions $d \ge 1$ via large deviation techniques. The initial aim is to prove that the asymptotics of $E^N_\beta$ are governed by the solution~$V_G^d$ to the variational problem $(VP\!:\!G)$:

 \begin{proposition}
 \label{p:crit}
Let $c > 0$ and $\beta \sim c /N$. Then the variational problem $(VP\!:\!G)$ has a finite solution $V_{G}^{d} (c)$ that is attained for a pair $f_1, f_2 \in C(\mathbb{R}^d)$. Moreover, as $N \to \infty$,
\[    \frac{1}{N^d} \log E^N_\beta =  d \log N +  V_G^d(c) + o(1)  . \]
 \end{proposition}
 
The secondary aim is to undertake an analysis of the function $V_G^d$, and in particular establish the properties listed in Proposition \ref{p:vprop}. In $d=1$ our analysis rests on the applicability of the classical Euler-Lagrange methods of the calculus of variations. This, again, is ultimately due to the Markov property of the Ornstein-Uhlenbeck process.

\subsection{Large deviation theory: Varadhan's lemma}

Our proof of Proposition \ref{p:crit} is based on an application of Varadhan's lemma in the setting of centred Gaussian measures on separable Banach spaces. Let us begin by recalling the relevant elements of the theory now.

\smallskip
Rather than work in the most general setting, let us specialise immediately to the case relevant to us. Recall from Section \ref{s:h} that $\Psi$ denotes the centred, almost surely continuous Gaussian field on $\mathbb{R}^d$ with covariance given by the Laplacian kernel \eqref{e:la}. Recall also from Section~\ref{s:vf} the entropy functional $I$ associated to $\Psi$. For each $c > 0$ we denote by $I_c$ the analogous entropy function on $f \in C[0, c]^d$, defined by
\[   I_c[f]  = \begin{cases}
\frac{1}{2} \|f\|_{\mathcal{H}_c} , & \text{if } f \in \mathcal{H}_c , \\
\infty, & \text{else}.
\end{cases}   \] 
Equip the set $C[0, c]^d$ with the usual topology generated by the sup-norm. Since this space is separable, it is well-known \cite{DV76} that $\Psi|_{[0, c]^d}$ satisfies a large deviation principle (LDP) with rate function given by $I_c$; more precisely,
\[   \limsup_{n \to \infty}  \frac{1}{n} \log \mathbb{P}(\Psi|_{[0, c]^d} /  \sqrt{n} \in \mathcal{A} ) \le  - \inf_{ f \in \mathcal{A} } I_c[f]     \ , \quad \mathcal{A}  \subset C[0, c]^d \text{ closed}   \]
and
\[   \liminf_{n \to \infty} \frac{1}{n} \log \mathbb{P}( \Psi|_{[0, c]^d} / \sqrt{n}  \in \mathcal{A} ) \ge  - \inf_{ f \in \mathcal{A} } I_c[f]     \ , \quad \mathcal{A}  \subset C[0, c]^d \text{ open}  . \]

It was shown in \cite[Section 3]{V66} (see also \cite[Theorem 1.4]{ER82}) that the existence of the LDP implies a version of Varadhan's lemma for sequences of converging functionals that satisfy a set of conditions. Again we state this result only in the case relevant to us, that of functionals on the product space $C[0,c]^d \times C[0,c]^d$:

\begin{theorem}[Varadhan's lemma; see {\cite[Section 3]{V66}}]
\label{t:vl}
Fix $c > 0$ and let $s_n$ be a scale satisfying $s_n \to \infty$ as $n \to \infty$. Let $X$ and $Y$ be two independent copies of $\Psi|_{[0, c]^d}$. Let $(F_n)_{n \in \mathbb{N}}$ and $F$ be measurable functions mapping $C[0, c]^d \times C[0, c]^d\mapsto \{-\infty\} \cup \mathbb{R}$ which satisfy the following two conditions:
\begin{enumerate}
\item As $n \to \infty$,
\[  F_n[f_n, g_n] \to F[f, g] \]
for all functions $(f_n)_{n \in \mathbb{N}}, (g_n)_{n \in \mathbb{N}}, f, g \in C[0, c]^d$ such that 
\[ \| f_n - f \|_\infty \to 0 \quad \text{and}  \quad \| g_n - g \|_\infty \to 0 ;  \text{ and } \] 
\item 
\begin{equation}
\label{e:tail}
 \lim_{L \to \infty} \limsup_{n \to \infty} \frac{1}{s_n} \log \mathbb{E} \left[ e^{s_n F_n[X / \sqrt{s_n}, Y / \sqrt{s_n} ] } \id_{F_n[X/ \sqrt{s_n}, Y / \sqrt{s_n} ] \ge L } \right]  = - \infty .
\end{equation}
\end{enumerate}
Then
\begin{equation}
\label{e:sup}
  \lim_{n \to \infty} \frac{1}{s_n} \log \mathbb{E}[ e^{s_n F_n[X/\sqrt{s_n}, Y/\sqrt{s_n}] } ]  =  \sup_{f_1, f_2 \in C[0, c]^d} F[ f_1, f_2 ] - I_c[f_1] - I_c[f_2] .
  \end{equation}
Moreover, if $F$ is lower semicontinuous, meaning that
\[   \liminf_{n \to \infty} F[f_n, g_n] \ge F[f, g] \]
for all functions $(f_n)_{n \in \mathbb{N}}, (g_n)_{n \in \mathbb{N}}, f, g \in C[0, c]^d$ such that 
\[ \| f_n - f \|_\infty \to 0 \quad \text{and}  \quad \| g_n - g \|_\infty \to 0 ,\] 
then the supremum in \eqref{e:sup} is finite and attained at some point in $C[0, c]^d \times C[0, c]^d$.
\end{theorem}

We now show how to apply Theorem \ref{t:vl} to extract the asymptotic growth-rate of $E^N_\beta$. Fix $c' > c > 0$ and assume that $\beta \sim c/N$. Recalling the set $\mathcal{T}^N :=  [\![1, N]\!]^d \subset \mathbb{Z}^d$, define $\mathcal{T}_\beta^N = \beta \mathcal{T}^N \subset \beta \mathbb{Z}^d$ and observe that, for sufficiently large $N$, $\mathcal{T}_\beta^N \subset [0, c']^d$; henceforth we will take $N$ sufficiently large such that this property holds. Define the functional $F_N : C[0, c']^d \times C[0, c']^d \to \{-\infty\} \cup \mathbb{R}$,
\[    F_N[f_1, f_2]  =  \frac{1}{N^d} \sum_{x \in \mathcal{T}_\beta^N}   \log(f_1(x)^2 + f_2(x)^2)  \]
(with the standard convention $\log 0 = -\infty$). Observe in particular that the domain of $F_N$ are functions in $C[0, c']^d$ and not $C[0, c]^d$; this is necessary since $\mathcal{T}_\beta^N$ may not lie in $[0, c]^d$ in general. Notice also that
\begin{align}
\label{e:e2}
 E^N_\beta  =  \mathbb{E} \big[ \prod_{x \in \mathcal{T}_\beta^N}  ( X^2_x + Y^2_x )   \big]  = e^{N^d  \log N^d} \, \mathbb{E} [ e^{  N^d F_N [X/ \sqrt{N^d} , Y / \sqrt{N^d} ] }  ] 
 \end{align}
where $X$ and $Y$ are independent copies of $\Psi|_{[0, c']^d}$.

\smallskip
Interpreting $F_N$ as a Riemann sum, define also a limiting version of $F_N$:
\[   F[f_1, f_2] = \frac{1}{c^d} \int_{s \in [0, c]^d} \log ( f_1(s)^2 + f_2(s)^2 )  \, ds  \ ,   \quad f_1, f_2 \in C[0, c']^d . \]
Since $\beta \sim c/N$ and $f_1, f_2$ are continuous, by the Riemann sum approximation of the Riemann integral on compact sets, the functionals $F_N$ converge to $F$ in the sense that the first condition of Theorem \ref{t:vl} holds. To check the second condition of Theorem \ref{t:vl}, we observe that, for all $\nu > 0$, there exists $C$ such that, for sufficiently large $N$,
\[  F_N[f_1, f_2 ]   \le 2 F[f_1,f_2] \le   \nu \left( \| f_1 \|_\infty^2 + \|f_2 \|_\infty^2 \right)   + C     \ , \quad f_1, f_2 \in C[0, c']^d .\]
Fixing $\nu > 0$ and defining 
\[ Z^2 :=  \nu \big( \|X\|_\infty + \|Y\|_\infty \big)^2 \ge \nu \big(\|X\|_\infty^2 + \|Y\|_\infty^2 \big) ,  \]
we see that \eqref{e:tail} is bounded above by
\[  \limsup_{L \to \infty} \limsup_{n \to \infty}  \frac{1}{s_n} \log \mathbb{E} \left[ e^{Z^2} \id_{ Z^2 \ge s_n (L-C) } \right] + C  . \]
Since $Z$ has a Gaussian upper tail (by the Borel-TIS inequality, for instance), we can choose $\nu > 0$ small enough so that $\mathbb{P}(Z^2 > z) \le e^{-2z}$ eventually as $z \to \infty$, which implies that 
\[ \limsup_{n \to \infty} \frac{1}{n} \log \mathbb{E} \left[ e^{Z^2} \id_{ Z^2 \ge n (L-C)} \right]   \le  - (L-C)  \]
as required. Finally, $F$ is readily seen to be lower semicontinuous (in fact continuous). Hence all the conditions of Theorem \ref{t:vl} are satisfied for this choice of $F_N$ and $F$.

\begin{proof}[Proof of Proposition \ref{p:crit}]
By equation \eqref{e:e2} and an application of Theorem \ref{t:vl} with the setting $s_n = N^d$, 
\begin{align*}  
 \lim_{N \to \infty} \left(  \frac{1}{N^d} \log E^N_\beta  - d \log N  \right) &=  \lim_{N \to \infty}   \frac{1}{N^d} \log  \mathbb{E} [ e^{  N^d F_N[ X/ \sqrt{N^d} , Y / \sqrt{N^d} ] }  ]   \\
 & =  \sup_{f'_1, f'_2 \in C[0, c']^d} F(f'_1, f'_2) - I_{c'}(f'_1) - I_{c'}(f'_2) , 
  \end{align*}
where the supremum is finite and attained for some $f'_1, f'_2 \in [0, c']^d$ (the attainment of the supremum is part of the conclusion of Theorem \ref{t:vl}). Recalling the energy functional $J_c$ from Section \ref{s:vf}, to complete the proof it remains only to show that
\[ \sup_{f'_1, f'_2 \in C[0, c']^d} F [f'_1, f'_2]  - I_{c'}[f'_1] - I_{c'}[f'_2]  = \sup_{f_1, f_2 \in C(\mathbb{R}^d) } J_c[ f_1, f_2] - I[f_1] - I[f_2] , \]
and that the supremum on the right-hand side is attained for some $f_1, f_2 \in C(\mathbb{R}^d)$. This follows from the fact that $F[f'_1, f'_2] = J_c[ f_1'|_{[0,c]^d}, f_2'|_{[0,c]^d}]$, and the restriction property of the norm $\|\cdot\|_\mathcal{H}$ stated in~\eqref{e:restrict}.
\end{proof}

\subsection{Analysis of the variational problem}

We proceed to analyse the function $V_G^d$, and in particular prove Proposition \ref{p:vprop}. Let us begin with some properties that are an immediate consequence of Proposition \ref{p:crit}:

\begin{lemma}
\label{l:prop2}
The function $V_G^d(c)$ is continuous, non-increasing and convex.
\end{lemma}  
\begin{proof}
By Proposition \ref{p:id} and Lemma \ref{l:convex}, the function 
\[ v_N:  c \mapsto \frac{1}{N^d} \log E^N_{c/N} - d \log N \]
 is strictly decreasing and convex for each $N \ge 2$. By Proposition \ref{p:crit}, $V_G^d$ is finite and is the point-wise limit of $v_N$ as $N \to \infty$, and hence $V_G^d$ is also non-increasing and convex. The continuity of $V_G^d$ then follows from the convexity.
\end{proof}

\begin{remark}
Note that we deduced the convexity of $V_G^d(c)$ from the convexity of $E^N_\beta$; it seems difficult to establish this property directly from the definition of $V_G^d$. On the other hand, we could have deduced the continuity of  $V_G^d(c)$  directly from its definition.
\end{remark}

Next we establish the basic asymptotic properties of $V_G^d$ that hold in all dimensions. In particular, we prove the following:

\begin{proposition}
\label{p:propvghigh}
The function $c \mapsto V_G^d(c)$ is continuous. Moreover, as $c \to 0$,
\[ V_G^d(c) = \log 2 - 1 + o(1),  \]
and, as $c \to \infty$,
\[ V_G^d(c) = - d \log c + O(1)  .\]
\end{proposition}

\begin{corollary}
\label{c:propvghigh}
The function $V_G^d(c)$ is strictly decreasing.
\end{corollary}

To prove Proposition \ref{p:propvghigh} we rely on some additional facts about the norm $\|\cdot\|_\mathcal{H}$. Recall the Laplacian kernel defined in \eqref{e:la} as
\[ \kappa(s, t) = e^{-|s-t|} .\]
Since $\kappa$ is stationary (i.e.\ $\kappa(s, t)$ depends only on $s-t$), we can define $\bar \kappa(s-t) = \kappa(s,t)$. The Fourier transform of $\bar \kappa$ lies in $L_2(\mathbb{R}^d)$, and is given by
\[ \rho(x) = c_0 \left( \frac{1}{1 + (2 \pi x)^2} \right)^{(d+1)/2} \]
for a normalising constant $c_0 > 0$ depending on the dimension; $\rho$ is sometimes called the \textit{spectral density} of $\Psi$. The standard theory of RKHS (see, e.g.,  \cite[Eq.(2.4), p.67]{bertho04}) then gives that, for all $f \in \mathcal{H}$,
\begin{equation}
\label{e:hform}
  \| f \|^2_\mathcal{H} =  \| \hat{f} / \sqrt{\rho} \|^2_2 = \frac{1}{c_0} \int_{x \in \mathbb{R}^d}    (1+ (2 \pi x)^2 )^{ (d+1)/2}  \hat{f}(x)^2 \, dx , 
  \end{equation}
where $\hat{f}$ denotes the Fourier transform of $f$. In the case $d = 1$, this leads directly to~\eqref{e:oun} since 
\[  \| f \|^2_\mathcal{H} =  \frac{1}{2} \int_{x \in \mathbb{R}}   \hat{f}(x)^2 + (2 \pi x \hat{f}(x))^2  \, dx  = \frac{1}{2} \left( \|f\|_2^2 + \|f'\|_2^2 \right) , \]
where in the last equality we applied Parseval's formula. One also sees that, in odd dimensions, $\|f\|_\mathcal{H}$ can be expressed as an integral over the first $(d+1)/2$ derivatives of $f$, which is related to the (pseudo-)domain Markov property mentioned in Remark \ref{r:dmp}.

\smallskip
The above representation of $\|\cdot\|_\mathcal{H}$ has the following easy consequences:

\begin{lemma}
\label{l:boundhigh}
There exists a number $c_1 > 0$ such that, for all $f \in \mathcal{H}$,
\[  \|f \|_\mathcal{H} \ge \| f \|_\infty \quad \text{and} \quad  \| f \|_\mathcal{H} \ge c_1 \|f\|_2  . \]
Moreover, for each $\delta > 0$ there exists a function $\tilde f \in \mathcal{H}$ satisfying, for all $c$ sufficiently small,
 \[   \tilde f|_{[0,c]^d} \ge 1  \quad \text{and} \quad \| \tilde f\|^2_\mathcal{H} \le 1 + \delta . \]
  Finally, there exists a number $c_2 > 0$ such that, for each $c > 1$, there is a function $\bar f \in \mathcal{H}$ satisfying 
 \[   \bar f|_{[0,c]^d}  \ge 1 \quad \text{and} \quad \| \bar f\|^2_\mathcal{H} \le c_2 c^d  . \]
\end{lemma}

\begin{proof}
By the reproducing property of the kernel \eqref{e:repro} and the Cauchy-Schwarz inequality
\[  f(x) = \langle f(\cdot), \kappa(x, \cdot) \rangle_{\mathcal{H}} \le \|\kappa(0, \cdot)\|_\mathcal{H}  \| f \|_\mathcal{H} = \| f \|_\mathcal{H} , \]
 where we used the fact that $\| \kappa(0, \cdot) \|_\mathcal{H} = 1$ (as can be seen from \eqref{e:rkhsip}). Moreover, since the spectral density $\rho$ is bounded above, by \eqref{e:hform} there is a $c_1 > 0$ such that
\[  \| f \|^2_\mathcal{H}  = \| \hat{f} / \sqrt{\rho} \|^2_2 \ge  c_1 \| \hat{f} \|^2_2    = c_1 \|f\|^2_2. \]

  \smallskip
For the second statement, let $\tilde f(s) = \sqrt{1+\delta} \kappa(0, s) = \sqrt{1 + \delta} e^{-|s|}$, which satisfies
  \[ \| \tilde f \|_\mathcal{H}  = \sqrt{1+\delta} \| \kappa(0, \cdot) \|_\mathcal{H}  = \sqrt{1 + \delta}, \]
  where again we used the fact that $\| \kappa(0, \cdot) \|_\mathcal{H} = 1$ by \eqref{e:rkhsip}. Setting $c > 0$ sufficiently small such that
  \[    e^{-2c} > 1/\sqrt{1+\delta} \]
  we have in addition that $\tilde f \ge \id_{[0,c]^d}$, as required.
  
  \smallskip
For the final statement, let $c_0 > 0$ be sufficiently large and $c_1 > 0$ be sufficiently small so that the Fourier transform of $c_0 \id_{[-c_1, c_1]^d}$ is larger than $1$ on $[0,1]^d$ (since the Fourier transform of a rectangular function is a product of sinc functions, such a choice is possible). Then, for any $c > 1$, let $\bar f$ be the Fourier transform of $c_0 c^d \id_{[-c_1/c, c_1/c]^d}$. By the scaling property of the Fourier transform, $\bar f|_{[0,c]^d} \ge 1$. Moreover, there is a $\delta > 0$ such that the spectral density $\rho$ on $\id_{[-c_1/c, c_1/c]^d}$ is bounded below by $\delta$. Hence, by \eqref{e:hform},
\begin{equation*}
  \| \bar f \|^2_\mathcal{H}  \le  \delta^{-1} c_0^2 c^{2d} \| \id_{[-c_1/c, c_1/c]^d} \|^2_2  \le  c_2 c^d .  \qedhere 
  \end{equation*}
\end{proof}

\begin{proof}[Proof of Proposition \ref{p:propvghigh}]

Recall the description of $V_G^d$ as the solution of the variational problem $(VP \!:\!G)$. We begin by establishing the upper bounds for the asymptotic statements. Since $\|f\|_\mathcal{H} \ge \|f\|_\infty$ by Lemma \ref{l:boundhigh}, we have
\[ V_G^d(c) \le  \sup_{f_1, f_2 \in C(\mathbb{R}^d)}   \log( \|f_1\|^2_\infty + \|f_2\|^2_\infty  )    -  \frac{1}{2} ( \|f_1\|^2_\infty + \|f_2\|^2_\infty ) \]
which is equal to
\[  \sup_{k_1, k_2 > 0}   \, \log ( k_1^2 + k_2^2)   - \frac{ 1}{2}  (k_1^2 + k_2^2) =  \sup_{k > 0}   \, \log k  - \frac{k}{2}     . \]
Optimising over $k > 0$, the maximum occurs at  $k = 2$ which yields the upper bound
\[    V_G^d \le  \log 2 - 1 . \]
Similarly, since $\| f \|_\mathcal{H} \ge c_1 \|f\|_2$, by  Lemma \ref{l:boundhigh}, we have
\[ V_G^d(c) \le  \sup_{f_1, f_2 \in C(\mathbb{R}^d)}    \frac{1}{c^d}  \int_{s \in [0, c]^d} \log( f_1(s)^2 + f_2(s)^2  )  \, ds -  \frac{c_1^2}{2}  \int_{s \in \mathbb{R}^d}  f_1(s)^2 + f_2(s)^2 \, ds  \]
which is equal to
\[    \sup_{k_1, k_2 > 0}   \, \log (k_1^2 + k_2^2)  - \frac{c_1^2 c^d}{2} ( k_1^2 + k_2^2)   = \sup_{k > 0}   \, \log k  - \frac{c_1^2 c^d}{2}  k    . \]
Optimising over $k > 0$, the maximum occurs at  $k = 2 c^{-d} / c_1^2$ which yields the upper bound
\[    V_G^d \le - d \log c + O(1) . \]

\smallskip
We turn now to the lower bounds. For the first statement, fix $\delta > 0$ and set $\tilde f$ as in Lemma~\ref{l:boundhigh}. Fix $k > 0$ and set $f_1 = k \tilde f$ and $f_2 = 0$. Then, by its definition as a supremum over all $f_1$ and~$f_2$, $V_G^d(c)$ is bounded below by 
\[    \int_{s \in [0, c]^d} \frac{1}{c^d} \log( k^2 \tilde f^2   )  \, ds -  \frac{1}{2} \| k \tilde f\|^2_\mathcal{H}   .\]
Since, for sufficiently small $c > 0$,
 \[   k \tilde f|_{[0,c]^d}  \ge k \quad \text{and} \quad \| k \tilde f\|^2_\mathcal{H} \le  k^2 (1+\delta)    ,\]
 the above is at least, for sufficiently small $c > 0$,
\[  2  \log k -  \frac{1+\delta}{2} k^2   .\]
Setting $k^2 = 2/(1+\delta)$, and using the inequality $\log(1+\delta) \le \delta$, yields the lower bound 
\[ V_G^d \ge  \log 2 - 1 -\log(1+\delta)   \ge \log 2 - 1 - \delta . \]
Since $\delta > 0$ was arbitrary, we have that $V_G^d \ge \log  2 - 1 +o(1)$ as $c \to 0$.

\smallskip
For the second statement, fix $c_2 > 0$ as in Lemma \ref{l:boundhigh}, and for each $c > 2$, let $\bar f$ be the function guaranteed to exist by the same lemma. Fix $k > 0$ and set $f_1 = k \bar f$ and $f_2 = 0$. Then, arguing as above, $V_G^d(c)$ is bounded below by 
\[    \int_{s \in [0, c]^d} \frac{1}{c^d} \log( k^2 \bar f^2   )  \, ds -  \frac{1}{2} \| k \bar f\|^2_\mathcal{H}   .\]
Since 
 \[   k \bar f|_{[0,c]^d} \ge 1   \quad \text{and} \quad \| k \bar f\|^2_\mathcal{H} \le c_2 k^2 c^d   ,\]
 the above is at least
\[  2  \log k -  \frac{c_2}{2} k^2 c^d   .\]
Setting $k = c^{-d/2}$ yields the lower bound $V_G^d \ge - d \log c + O(1)$.
\end{proof}

\begin{proof}[Proof of Corollary \ref{c:propvghigh}]
The convexity of $V_G^d$ (by Lemma~\ref{l:prop2}) implies that $V_G^d$ may only fail to be strictly decreasing if it is constant on $[c, \infty]$ for some $c > 0$. However, this contradicts the fact that $V_G^d(c) \to - \infty$ as $c \to \infty$, as shown in Proposition \ref{p:propvghigh}.
\end{proof}

We turn now to the one-dimensional case in which a more detailed analysis can be carried out. In particular, we have the following ODE representation:

\begin{proposition}
\label{p:ode}
For each $c> 0$, the function $V_G^1$ has the representation
    \begin{equation}
    \label{e:eqforv}
     V_G^1 =  \log 2 -  h_c(0)^2 + \int_{s \in [0, c]}  \frac{2}{c}  \log h_c(s) - \frac{1}{2} \left( h_c(s) +  \dot{h}_c(s)  \right)^2  \, ds  , 
     \end{equation}
     where $h_c$ is the positive analytic function on $[0,c]$ that is the unique solution to the ODE in~\eqref{e:ode}.
     \end{proposition}

\begin{proof}
By the definition of $J_c$ and the representation of $I_c$ in \eqref{e:oun2}, $V_G^1(c)$ can be written as 
\begin{align*}
 & \sup_{f_1, f_2 \in \mathcal{H}_c}   \bigg\{ -  \frac{1}{4} ( f_1(0)^2 + f_1(c)^2 + f_2(0)^2 + f_2(c)^2 )   \\
  & \qquad  \qquad +  \int_{s \in [0,c]}  \frac{1}{c} \log ( f_1(s)^2 + f_2(s)^2 )- \frac{1}{4} (  f_1(s)^2 +  \dot{f}_1(s)^2 + f_2(s)^2 +  \dot{f}_2(s)^2 )   \, ds \bigg\}   ,
  \end{align*}
  where $\mathcal{H}_c$ is the space of absolutely continuous functions on $[0, c]$ with derivative in $L^2[0, c]$. Switching to polar coordinates 
  \[ (f_1(s), f_2(s)) = \big(r(s) \cos \theta(s), r(s) \sin \theta(s) \big) , \]
  we see that the expression to be maximised depends on $\theta(s)$ only through
  \[  -  ( \dot{f}_1(s)^2 +  \dot{f}_2(s)^2)  = - r(s)^2 \dot{\theta}(s)^2 - \dot{r}(s)^2  ,  \]
  which is a decreasing function of $\dot{\theta}$. Hence any maximiser must have $\theta$ constant, so it remains to analyse the one-dimensional variational problem
    \begin{equation}
    \label{e:polar}
       \sup_{ r \in \mathcal{H}_c ,  r > 0  }  - \frac{r(0)^2 + r(c)^2}{4} + \int_{s \in [0, c]} \frac{2}{c}  \log r(s)  - \frac{1}{4} \left( r(s)^2 + \dot{r}(s)^2  \right) \, ds  . 
       \end{equation}
  Fixing for a moment $r(0)$ and $r(c)$, and widening the class of admissible $r$ to all positive absolutely continuous functions on $[0, c]$, the above problem is in the standard setting of the calculus of variations (see, e.g., \cite{dac}) since the Lagrangian
  \[ L( r, \dot{r} ) = - \frac{2}{c} \log  r  + \frac{1}{4} ( r^2 + \dot{r}^2 )     \]
  is smooth, strictly convex with non-degenerate second derivatives, and is coercive.\footnote{I.e.\ there exist $\alpha_1, \alpha_2 > 0$, $\alpha_3 \in \mathbb{R}$ and $p_2 > p_1 \ge1$ such that $L( u, \dot{u} )\ge \alpha_1 |u|^{p_1} +  \alpha_2 |\dot{u}|^{p_2} - \alpha_3$.} The standard theory (see, e.g., \cite[Theorems 4.1 and 4.36]{dac}) yields the existence of a smooth and positive $r$ that maximises the variational problem and satisfies the Euler-Lagrange equation; being smooth, this solution also maximises the variational problem in the space~$\mathcal{H}_c$.
  
  \smallskip
The Euler-Lagrange equation for $r$ reads
    \[    \frac{\partial}{\partial r} \left(  \frac{2}{c} \log r(s) - \frac{1}{4} r(s)^2  \right)    = \frac{d}{d s} \left(   \frac{\partial}{\partial \dot{r}} \left( - \frac{1}{4} \dot{r}(s)^2 \right) \right)   \]
with boundary conditions (which arise from the contribution of the endpoints $ - \frac14  (r(0)^2 + r(c)^2)$)    
   \[   \frac{\partial}{dr} \left( -\frac{r(s)^2}{4} \right) \bigg|_{s = 0} =  \frac{\partial}{d \dot{r}} \left( -\frac{\dot{r}(s)^2}{4} \right) \bigg|_{s = 0}  \quad \text{and} \quad  \frac{\partial}{dr} \left( -\frac{r(s)^2}{4} \right) \bigg|_{s = c} = -  \frac{\partial}{d \dot{r}} \left( -\frac{\dot{r}(s)^2}{4} \right) \bigg|_{s = c}  . \]
        Simplifying, one sees that $r$ satisfies the boundary value ODE
      \begin{equation}
      \label{e:ode3}
        \ddot{r} = - \frac{4}{c r}  + r \ ,  \quad \dot{r}(0) = r(0) \ , \ \dot{r}(c) = -r(c) .
        \end{equation}
      Assuming for a moment that the solution to this ODE is unique, $r$ must be symmetric under $x \mapsto c-x$, since $r(x-c)$ also satisfies this ODE. Setting $y = r / \sqrt{2}$ then recovers the ODE in \eqref{e:ode}, and substituting into \eqref{e:polar} establishes \eqref{e:eqforv}.
    
        \smallskip     
    It remains to show that \eqref{e:ode3} (and hence also \eqref{e:ode}) has a unique solution that is analytic. To show uniqueness, consider that $\ddot{r}$ is strictly increasing in $r$, and since also $\dot{r}(0)$ is increasing in $r(0)$, each of $r, \dot{r}$ and $\ddot{r}$ are strictly increasing in~$r(0)$. Hence the quantity
     \[  Z(r(0)) =  \left( \dot{r}(c) + r(c) \right)  - \left( \dot{r}(0) - r(0) \right)  = r(0) + r(c) + \int_{0}^{c} \ddot{r}(s) \, ds \]
     is strictly increasing in $r(0)$, and since the solution satisfies $Z(r(0)) = 0$, it is unique. Analyticity follows from the Cauchy--Kowalevski theorem.
        \end{proof}
         
  We can exploit the ODE representation in Proposition \ref{p:ode} to show that $V_G^1$ is a smooth function of $c$:
  
  \begin{proposition}
 \label{p:smooth}        
Both $h_c$ and $\dot{h}_c$ vary smoothly with $c$, and hence $V_G^1(c)$ is smooth.
\end{proposition}

\begin{proof}
Multiplying \eqref{e:ode} with $\dot{y}$ and integrating from zero gives
     \begin{equation}
     \label{e:ode2}
          \dot{y}^2 = y^2 - \frac{4}{c} \log(y / y(0) )  \ , \quad y(0) = y(c) .
          \end{equation}
    To analyse \eqref{e:ode2}, recall that we previously argued in the proof of Proposition \ref{p:ode} that the unique solution $h_c$ is symmetric under $x \mapsto c-x$, and also that $\ddot{h}_c$ is increasing in $h_c$, which implies that $h_c$ is concave since necessarily $\ddot{h}_c \le 0$ at the global maximum. Hence $h_c$ can be represented implicitly on $x \in [0, c/2]$.
    
  \smallskip
  Next, we introduce the parameters $(a_1, a_2) = (h_c(0), h_c(c/2))$ which, by solving \eqref{e:ode2}, determine the implicit function representation for $h_c$ via
     \[      \int_{a_1}^{h_c(x) } \frac{dy}{\sqrt{y^2 - \frac{4}{c} \log(y / a_1 ) } }  = \begin{cases}  x ,  & x \in [0, c/2]   , \\ c-x,& x \in [c/2,c] .\end{cases} \]
  From \eqref{e:ode2}, $(a_1, a_2)$ is the unique solution in $\mathbb{R}_+^2$ of the smooth non-linear system
     \[   G_1(a_1, a_2; c) = 0 \quad \text{and} \quad G_2(a_1, a_2; c) = 0  \]
    where
    \[ G_1(a_1, a_2 ; c) = a_2^2 - \frac{4}{c} \log(a_2/a_1) \quad \text{and} \quad G_2(a_1, a_2;c) = -\frac{c}{2} + \int_{a_1}^{a_2} \frac{dy}{\sqrt{y^2 - \frac{4}{c} \log(y / a_1 ) } }  . \]
 Finally, since $G_1$ and $G_2$ vary smoothly with $c$, uniqueness implies that the parameters $(a_1, a_2)$ do also, and hence so does the solution $h_c$. Given \eqref{e:ode2}, we see that $\dot{h_c}$ also varies smoothly with $c$, and hence so does $V_G^1$.
         \end{proof}

We can now complete the proof of Proposition \ref{p:vprop}:

\begin{proof}[Proof of Proposition \ref{p:vprop}]
In light of Lemma~\ref{l:prop2} and Propositions \ref{p:propvghigh},~\ref{p:ode} and \ref{p:smooth} and Corollary \ref{c:propvghigh}, it remains only to show that
 \[   V_G^1(c)   = \log 2 - 1 -c/3 + O(c^2) \quad \text{as } c \to 0 . \]
  
\smallskip
Let $h$ be the unique solution to the ODE \eqref{e:ode} (dropping the explicit dependence on $c$ for simplicity), and recall that $h\ge 0$ and, from the proof of Proposition~\ref{p:ode}, that $h$ is symmetric under $x \mapsto c-x$. We begin with an \textit{a priori} estimate on $h(0)$. As in the proof of Proposition~\ref{p:smooth}, $h$ is concave, and so $|\dot{h}| \le \dot{h}(0) = h(0)$. We deduce that $h \in [h(0), h(0) + c h(0)]$, which implies the following bounds on $\ddot{h}$:
 \begin{equation}
 \label{e:boundhdd}
   -\frac{2}{c h(0) } + h(0) \le \ddot{h} \le - \frac{2}{ c h(0)(1 + c) } + h(0)(1 + c) .  
   \end{equation}
Given the boundary conditions in \eqref{e:ode},
 \[ \int_{s \in [0, c]} \ddot{h}(s) = \dot{h}(c) - \dot{h}(0) = -2 h(0) ,\] 
 and so \eqref{e:boundhdd} yields
 \[      -\frac{2}{c h(0) } + h(0) \le -\frac{2}{c} h(0) \le - \frac{2}{ c h(0)(1 + c) } + h(0)(1 + c) .  \]
 Simplifying, this demonstrates that 
 \[   \frac{1}{(1 + c)(1 + c(1 + c) )  }  \le   h(0)^2 \le  \frac{1}{1 + c}  \] 
 and in particular
  \begin{equation}
 \label{e:boundh}
 h(0) = 1 + O(c)  \ , \quad \text{as } c \to 0 .
 \end{equation}

 By \eqref{e:ode} and \eqref{e:boundh} we have that, as $c \to 0$,
 \[ \ddot{h}(0) = -2 h(0) /c + O(1)  . \]
  By differentiating \eqref{e:ode} and since $|\dot{h}| \le h(0)$, one can also deduce from \eqref{e:boundhdd} and \eqref{e:boundh} that, for sufficiently small $c > 0$, $|\dddot{h}| \le 3/c$. Since $h$ is analytic, this implies the following expansions for $h$ and $\dot{h}$ that are uniform on $s\in [0, c]$:
  \begin{align*}
     h(s) & = h(0) + s \dot{h}(0) + \frac{s^2}{2} \ddot{h}(0) + O(c^2)  =  h(0) \left( 1 +  s -\frac{s^2}{c} \right) +  O(c^2)  ,  \\
     \dot{h}(s) & = \dot{h}(0) + O(c) .
     \end{align*}
 Putting this into \eqref{e:eqforv}, and using a Taylor expansion of the logarithm,
 \begin{align*}
  V_G^1(c) & = \log 2  + 2 \log h(0)  -   h(0)^2  \\
 & \quad  +  \int_{s \in [0, c]}  \frac{2}{c} \log \left( 1 + s  -\frac{s^2}{c}    + O(c^2)  \right)  - \frac{1}{2} h(0)^2 \left(   2   -\frac{2s}{c}  + O(c)      \right)^2  \, ds  \\
  &   = \log 2  + 2 \log h(0)  -   h(0)^2  \\
 & \quad +  \int_{s \in [0, c]}  \frac{2s}{c} - \frac{2 s^2}{c^2} - 2 + \frac{4s}{c} - \frac{2 s^2}{c^2}   + O(c)  \, ds  \\
 & = \log 2  + 2 \log h(0)  -   h(0)^2    + \Big(3  - \frac{4}{3} - 2   \Big) c + O(c^2) .
 \end{align*}
 Since $h(0) = 1 + t_c$ for some $t_c = O(c)$, and again using a Taylor expansion of the logarithm,
 \begin{equation*}
 V_G^1(c) = \log 2  - 1 + 2 t_c - 2 t_c +  \Big(3  - \frac{4}{3} - 2   \Big) c + O(c^2)   =  \log 2 - c/3 + O(c^2)  . \qedhere
 \end{equation*}
\end{proof}


\medskip

\section{The subcritical regime in arbitrary dimension}
\label{s:subhigh}

In this section we use elementary Gaussian estimates to study the subcritical regime in arbitrary dimension $d \ge 1$.

 \begin{proposition}
 \label{p:high}
Let $1/N \ll \beta \ll 1$. Then
 \[ \frac{1}{N^d} \log  E^N_\beta  =   \log(1/\beta^d) + O(1)  .  \]
 \end{proposition}
 
Before we prove Proposition \ref{p:high} we state two auxiliary lemmas. The first is a simple bound on the permanent of a positive matrix, which can be directly verified from the definition. The second is a general `persistence' bound for positively-correlated Gaussian fields.
 
 \begin{lemma}
 \label{l:permbound}
Let $A$ be a positive $N \times N$ matrix. Then
\[ \text{perm}(A)  \le  \prod_i r_i  , \]
where $r_i$ denotes the row-sum of the $i^{\rm{th}}$ row of $A$.
\end{lemma}

\begin{lemma}
\label{l:pers}
Let $D \subset \mathbb{R}^d$ be a compact domain, and let $\Psi$ be an almost surely continuous centred Gaussian field on $\mathbb{R}^d$ with covariance kernel satisfying
\[ \kappa(s, t) = \mathbb{E}[\Psi(s)\Psi(t)] > 0 \]
for all $s, t \in D$. Then there exists a $c > 0$ such that as $t \to \infty$ eventually 
\[ \mathbb{P}[ \Psi|_D > t ] > e^{- c t^2  }  . \]
\end{lemma}
\begin{proof}
Without loss of generality we may assume that $\Psi$ has unit variance and that $\{0 \} \in D$. Since $D$ is compact, we may define
\[ c_1 =  \inf_{s \in D} \kappa(0, s)  > 0 . \]
By Gaussian regression, conditionally on the value of $\Psi(0) = \ell$, 
\begin{equation*}
   \Psi(s) = \ell \kappa(0, s)  + h(s)  ,
   \end{equation*}
where $h$ is an independent almost surely continuous centred Gaussian field. Hence, given the definition of $c_1$, the event $\{\Psi|_D > t\}$ is implied by
\[   \{ \Psi(0) > 2t / c_1 \} \cap \{ \|h\|_\infty < t \}  , \]
and so
\[ \mathbb{P}[ \Psi|_D > t ]  \ge \mathbb{P}( \|h\|_\infty  < t ) \cdot \left(1 - \Phi(2t /c_1) \right) ,   \]
where $\Phi$ denotes the standard Gaussian cdf. To finish, since $ \|h\|_\infty $ is almost surely finite,
\[   \mathbb{P}( \|h\|_\infty < t )  > 1/2     \]
 for sufficiently large $t$, and we get the result.
\end{proof}

\begin{proof}[Proof of Proposition \ref{p:high}]

We begin with the upper bound. Applying Lemma \ref{l:permbound},
\[    \frac{1}{N^d} \log  E^N_\beta  = \frac{1}{N^d} \log  \text{perm}(A^N_\beta) + \log 2 \le \max_i \log r_i + O(1) ,   \]
where $r_i$ denotes the row-sum of the $i^{\rm{th}}$ row of $A$. The row sums of $A^N_\beta$ are bound above by the Riemann sum
\[   R_\beta =  \sum_{z \in \beta \mathbb{Z}^d} e^{-|z| } . \]
Since the function $z \mapsto e^{-|z|}$ is integrable, the Riemann sum $R_\beta$ satisfies, as $\beta \to 0$,
\[ R_\beta = \Theta(\beta^{-d}) , \]
which gives the result.

Turning now to the lower bound, recall that $\mathcal{T}^N_\beta = \beta  [\![ 0, N-1 ]\!]^d \subset \beta \mathbb{Z}^d$, and partition $\mathcal{T}^N_\beta$ into at most $2 (\beta N)^d$ disjoint sets $S_i$ all of diameter at most $2$. Define a matrix $\bar A^N_\beta$ via
\[ (\bar A^N_\beta)_{s,t}  = \begin{cases} 
(A^N_\beta)_{s, t},  &  \text{if } s, t \text{ belong to a common set } S_i,  \\
0, & \text{else.}
\end{cases} \]
Then by the definition of the permanent, 
\[   \frac{1}{N^d} \log  E^N_\beta  = \frac{1}{N^d} \log  \text{perm}(A^N_\beta)  + \log 2 \ge \frac{1}{N^d} \log  \text{perm}(\bar A^N_\beta) + \log 2. \] 
Recall that $\Psi$ denotes the stationary, almost surely continuous, centred Gaussian field on $\mathbb{R}^d$ with the Laplacian covariance kernel \eqref{e:la}. Observe that by Lemma \ref{l:reed}
\[    \text{perm}(\bar A^N_\beta) =  2^{-N^d} \prod_{i }  \mathbb{E} \Big[     \Pi_{ x \in S_i } (X_x^2 + Y_x^2)   \Big]  ,  \]
recalling also that, for a Gaussian vector, coordinates being uncorrelated are equivalent to them being independent. Hence we deduce that
\[   \frac{1}{N^d} \log  E^N_\beta   \ge \frac{1}{N^d}  \sum_i   \log \mathbb{E} \Big[ \Pi_{x \in S_i} (X_x^2 + Y_x^2 )  \Big]  \ge \frac{1}{N^d}  \sum_i   \log \mathbb{E} \Big[ \Pi_{x \in S_i} X_x^2   \Big] .  \]
Applying Lemma \ref{l:pers}, and since the diameters of $S_i$ are bounded and $\Psi$ is stationary, there is a $c > 0$ such that, for all sufficiently small $\beta > 0$ and all $i$,
\[
  \mathbb{E} \Big[   \Pi_{x \in S_i} X_x^2   \Big]    \ge \mathbb{E} \big[   \Pi_{x \in S_i } X_x^2   \: \big| \:  X|_{S_i}  > \beta^{-d/2} \big] \, \mathbb{P} \big[ X|_{S_i}  > \beta^{-d/2}\big]   \ge  (1/\beta^d)^{|S_i|} e^{-c\beta^{-d} }.
\]
 Hence
\[ \frac{1}{N^d} \log  E^N_\beta \ge \frac{1}{N^d} \log(1/\beta^d) \sum_i |S_i|   - \frac{c}{(\beta N)^d}  \# \{ S_i \}  ,\]
and since $\sum_i |S_i| = N^d$ and $ \# \{ S_i \}  \le 2(\beta N)^d$, we conclude that
 \begin{equation*}
 \frac{1}{N^d} \log  E^N_\beta \ge \log(\beta^{-d}) + O(1) . \qedhere
\end{equation*}
\end{proof}


\medskip

\section{Completing the proof of the main results}
\label{s:final}

In this section we complete the proof of Theorems \ref{t:main1}, \ref{t:main2} and \ref{t:part}. We make use of two elementary lemmas, which show how bounds on sequences of convex functions may be transferred to their derivatives.

\begin{lemma}
\label{l:convex1}
Let $h_n : [0, \infty) \to \mathbb{R}$ be a sequence of continuously differentiable and convex functions converging to a limit $h(x) := \lim_{n \to \infty} h_n(x)$.
\begin{enumerate}
\item The limit $h$ is convex and hence differentiable except on a countable set. At each point $x > 0$ such that $h$ is differentiable,
\[ h'(x) = \lim_{n \to \infty} h'_n(x) .\]
\item Suppose $h_n(0) = 0$ for all $n$, $h$ is continuously differentiable at $0$, and 
\[ h'(0) = \lim_{n \to \infty} h'_n(0) .\]
Then for every sequence $x_n \to 0$, as $n \to \infty$,
\[ h_n(x_n) = h'(0) x_n + o(x_n)  \qquad \text{and} \qquad h_n'(x_n) = h'(0) + o(1) . \]
\end{enumerate}
\end{lemma}

\begin{lemma}
\label{l:convex2}
Let $h:  \mathbb{R}^+ \to \mathbb{R}$ be a continuously differentiable convex function. Let $x \in (0, \infty)$ and suppose there exist constants $c_1, c_2 > 0$ such that
\[  - c_2  \le \inf_{y \in [x/2, x]} \Big( h(y) + c_1 \log y \Big) \le \sup_{y \in [x, 2e^{2c_2/c_1} x]} \Big( h(y) + c_1 \log y \Big) \le  c_2  .  \]
Then 
\[  - \frac{c_1 }{ x} \times (1 + c_3)  \le h'(x) \le   - \frac{c_1 }{ x}  \times (1 - c_3) \]
where $c_3 \in (0, 1)$ depends only on $c_1$ and $c_2$. Moreover if $c_2 < c_1/4$ then one may take
\[ c_3 = 2 \sqrt{c_2/c_1}  .  \]
\end{lemma}

We defer the proof of Lemmas \ref{l:convex1} and \ref{l:convex2} to the end of the section. Let us complete the proof of the main results.

\begin{proof}[Proof of Theorem \ref{t:part}]
In the subcritical regime $1/N \ll \beta \ll 1$ the result follows by combining Proposition \ref{p:id}, with Proposition \ref{p:high} (in the general case $d \ge 1$) or Proposition~\ref{p:noncrit} (in the case $d = 1)$. 
 
Let us treat the critical regime $\beta \sim c / N$. Consider the sequence of smooth non-increasing convex functions
\[ h^d_N(c) :=  \frac{1}{N^d} \log Z_N(c/N)  - \frac{1}{N^d} \log (( N^d)!)   \  , \quad c \ge 0 , \]
which satisfy $h^d_N(c) \le 0$ and $h^d_N(0) = 0$. By Propositions \ref{p:id}, \ref{p:vprop} and \ref{p:crit}, the sequence $h^d_N$ converges, as $N \to \infty$, to a continuous, strictly decreasing, convex function $g^d :[0,\infty) \to \mathbb{R}$ whose restriction to $(0, \infty)$ satisfies the conditions of Theorem \ref{t:part}. Since $g^d$ is continuous and $c \mapsto h^d_N(c)$ is monotone, the convergence is uniform on compact sets, and this completes the proof in the critical regime. 

Finally we turn to the supercritical regime $\beta \ll 1/N$. In the general case $d \ge 1$, by the uniform convergence of $h^d_N \to g^d$ we have
\[ h^d_N(c_N) \to  g^d(0) = 0 \]
for any sequence $c_N \to 0$. In the case $d = 1$, recall that $g^1(c)$ is smooth on $c > 0$, and $g^1(c) =  -c/3 + O(c^2)$ as $c \to 0$, and so in particular $\frac{d}{d c} g^1|_{c=0} = -1/3$. We also note that (see Proposition \ref{p:id} and Remark \ref{r:hyper})
\[ \frac{\partial}{\partial c}  h^1_N(c) \Big|_{c = 0} =  \frac{1}{N}   \frac{\partial}{\partial \beta}  \Big(\frac{1}{N}   \log Z_N(\beta) \Big) \Big|_{\beta = 0} = - \frac{1}{N}  D_0^N  \to  -1/3  = \frac{d}{d c} g^1(c) \Big|_{c=0} . \]
Hence by item (2) of Lemma \ref{l:convex1},
\[ h^1_N(c_N) = - c_N / 3 + o(c_N) \]
for any sequence $c_N \to 0$, which completes the proof.
\end{proof}

 \begin{proof}[Proof of Theorem \ref{t:main2}]
 Let us first consider the critical regime $\beta \sim c / N$. Recall the sequence of smooth non-increasing convex functions
\[ h^1_N(c) :=  - \frac{1}{N} \log Z_N(c/N)  - \frac{1}{N} \log (( N)!)   \to g^1(c) \  , \quad c \ge 0 , \]
from the proof of Theorem \ref{t:part}. Define $f^1(c) := -\frac{d}{dc} g^1(c)$, and observe that $f^1$ satisfies all the conditions of Theorem \ref{t:main2} (the fact that $f^1(c)  = \Theta(1/c)$ as $c \to \infty$ follows by applying Lemma \ref{l:convex2} to the function $g^1$,  since $g^1(c) = -  \log c + O(1)$ as $c \to \infty$). Now by Proposition \ref{p:id}, 
\[ \frac{\partial}{\partial c} h^1_N(c) = -\frac{1}{N} D_{c/N}^N  ,\]
and so, by item (1) of Lemma \ref{l:convex1}, for any $c > 0$
\[ \frac{1}{N} D_{c/N}^N  \to  - \frac{d}{d c} g^1(c) = f^1(c) .\]
Since $f^1$ is continuous and $c \mapsto \frac{1}{N} D_{c/N}^N$ is monotone, the convergence is uniform on a neighbourhood of $c$, which completes the proof in the critical regime.

Turning to the supercritical regime $\beta \ll 1/N$, as in the proof of Theorem \ref{t:part} we have that
\[ \frac{\partial}{\partial c}  h^1_N(c) \Big|_{c = 0} =  - \frac{1}{N}  D_0^N  \to  -1/3 = \frac{d}{d c} g^1(c) \Big|_{c=0} . \]
Hence by item (2) of Lemma \ref{l:convex1},
\[  \frac{1}{N}  D_{\beta_N}^N = -\frac{\partial}{\partial c} h^1_N(c) \Big|_{c = \beta_N N}   =  1 / 3 + o(1) \]
for any sequence $\beta_N \ll 1/N$, which completes the proof in this regime.
  
Finally we consider the subcritical regime $1/N \ll \beta \ll 1$, in which it is more convenient to work with the functions
\[ \bar{h}_N(\beta) :=  \frac{1}{N^d} \log Z_N(\beta) - \log 2  + 1    \  , \quad \beta \ge 0 , \]
which satisfy, by Proposition \ref{p:id}, $\frac{\partial}{\partial \beta}  \bar{h}_N(\beta) = - D_{\beta}^N$. Recall that Theorem \ref{t:part} states that 
\[ \bar{h}_N(\beta)  = - \log \beta + O(\beta) + O(1/(\beta N)) \]
for any sequence $1/N \ll \beta \ll 1$. Applying Lemma \ref{l:convex2},
\[  \frac{\partial}{\partial \beta}  \bar{h}_N(\beta) = - 1/\beta \times (1 + O(\beta^{1/2}) + O((\beta N)^{-1/2} )  \]
as required.
\end{proof}

 \begin{proof}[Proof of Theorem \ref{t:main1}]
The results in the critical regime $\beta = c/N$ follow as in the proof of Theorem \ref{t:main2} (although if $d \ge 2$ we know only that $\frac{1}{N} D_{c/N}^N  \to  - \frac{d}{d c} g^d(c) = f^d(c)$ outside a countable set $\mathcal{C}$, and the convergence may not be uniform so we cannot extend the results to $\beta \sim c/N$). 

In the subcritical regime $1/N \ll \beta \ll 1$ we consider the functions
\[ \bar{h}^d_N(\beta) :=  \frac{1}{N^d} \log Z_N(\beta)    \  , \quad \beta \ge 0 ,\]
which satisfy, by Proposition \ref{p:id}, $\frac{\partial}{\partial \beta}  \bar{h}^1_N(\beta) = - D_{\beta}^N$. Recall that Theorem \ref{t:part} states that
\[ \bar{h}^d_N(\beta)  = - d \log \beta + O(1) \]
for any sequence $1/N \ll \beta \ll 1$. Hence, by Lemma \ref{l:convex2},
\[ \frac{\partial}{\partial \beta}  \bar{h}^1_N(\beta) = - \Theta(1/\beta) \]
as required.

The result in the supercritical regime $\beta \ll 1/N$ then follows since $\beta \mapsto D_\beta^N/N$ is non-increasing and $D_0^N/N \le \sqrt{d}$ by definition. 
 \end{proof}
 
 \begin{remark}
We emphasise that in order to obtain deduce, in dimension $d=1$, the asymptotics
\[\frac{1}{N} \log Z_N(\beta) =  \frac{1}{N^d} \log (N!) - \beta N / 3 + o(\beta N) \quad \text{and} \quad D_{\beta}^N   =  N / 3 + o(1) \]
 in the supercritical regime $\beta_N \ll 1/N$ from those in the critical regime $\beta \sim c/N$, it was crucial that we could prove that
\[  \lim_{c \to 0} f^1(c) :=  \lim_{c \to 0} - \frac{d}{d c} g^1(c)  = 1 / 3  \]
is equal to the `hypercube line picking constant'  $c_1 := \lim_{N \to \infty} \frac{1}{N}  D_0^N$ (see Remark~\ref{r:hyper}), since this allowed us to apply item (2) of Lemma \ref{l:convex1}. We do not know whether this is true for general $d \ge 2$.
\end{remark}

We conclude the section by proving Lemmas \ref{l:convex1} and \ref{l:convex2}. For this we use the following elementary lemma:

 \begin{lemma}
\label{l:ra}
Let $I \subset \mathbb{R}$ be an open interval, let $h: I \to \mathbb{R}$ be continuously differentiable and convex, and let $h^-, h^+: I \to \mathbb{R}$ be such that $h^- \le h \le h^+$. Then, for each $x \in I$,
\[h'(x) \in  \left[   \sup_{  \delta > 0 : x-\delta \in I} \frac{ h^-(x) - h^+(x - \delta) }{\delta}  ,  \inf_{\delta > 0 : x + \delta \in I} \frac{ h^+(x + \delta) - h^-(x)  }{\delta}  \right] . \]
\end{lemma} 
\begin{proof}
We prove the upper bound; the proof of the lower bound is similar. Suppose for contradiction that for some $\delta > 0$
\[  h'(x) > \frac{ h^+(x + \delta) - h^-(x)  }{\delta}  .\]
Then, since $h$ is convex (and so in particular $h'$ is non-decreasing), 
\[    h(x + \delta)  - h(x) \ge \delta h'(x) > \delta \times \frac{ h^+(x + \delta) - h^-(x)  }{\delta}  = h^+(x + \delta) - h^-(x)   ,\]
which is in contradiction with $h^- \le h \le h^+$.
\end{proof}
 
 \begin{proof}[Proof of Lemma \ref{l:convex1}]
 \noindent Item (1). It is immediate that the limit $h$ is convex. To prove the convergence $h_n' \to h'$, fix $x  > 0$ and $\delta, \varepsilon \in (0, x)$. Since $h_n \to h$ we have that
 \[   h(y) - \varepsilon \le h_n(y) \le h(y) + \varepsilon  \ , \quad \text{for } y \in \{x - \delta, x , x + \delta \} \]
 for all $n$ sufficiently large. Applying Lemma \ref{l:ra}, 
\[h'_n(x)   \in  \left[  \frac{h(x) - h(x-\delta) - 2\varepsilon}{\delta}  ,   \frac{h(x+\delta) - h(x) + 2\varepsilon}{\delta}  \right] \]
for all $n$ sufficiently large. Sending $\varepsilon \to 0$ and then $\delta \to 0$ yields the result.
 
\noindent Item (2). By assumption $h'$ is continuous and $h_n' \to h'$ in a neighbourhood of $0$. Since $x \mapsto h_n'(x)$ is monotone the convergence is uniform, which gives the second assertion. The first assertion then follows by integrating $h_n$ over $[0, x_n]$.
\end{proof}

\begin{proof}[Proof of Lemma \ref{l:convex2}]
Let $x > 0$ and $c_1, c_2 > 0$ be as in the statement of the lemma. Applying Lemma \ref{l:ra}, 
\begin{equation}
\label{e:convex2}
 h'(x) \in      \left[ -\frac{c_1}{x} \times \Big( \frac{ - \log (1- \delta_1)  + 2c_2/c_1}{\delta_1} \Big)  ,  -\frac{c_1}{x} \times \Big( \frac{  \log (1 + \delta_2)  - 2c_2/c_1}{\delta_2} \Big) \right]  
 \end{equation}
for any $\delta_1 \in (0,1/2]$ and $\delta_2 \in (0, 2 e^{2c_2/c_1} - 1]$. Choosing $\delta_1 = 1/2$ and $\delta_2 = 2 e^{2c_2/c_1} - 1$ 
gives that
\[  h'(x) \in      \left[ -\frac{c_1}{x} \times \Big( 2 \log 2  + 4c_2/c_1 \Big)  ,  -\frac{c_1}{x} \times \Big( \frac{  \log 2 }{ 2 e^{2 c_2/c_1} - 1} \Big) \right]  \]
which proves the first assertion.

For the second assertion, recall that 
\[ \log(1-x)  \le  - x    \quad \text{and} \quad \log(1+x) \ge x - x^2/2 , \]
for all $x > 0$. Inputting these bounds into \eqref{e:convex2} gives
\[  h'(x) \in      \left[ -\frac{c_1}{x} \times \Big( 1 + 2c_2/(c_1 \delta_1) \Big)  ,  -\frac{c_1}{x} \times \Big( 1  - \delta_2/2 - 2c_2/(c_1 \delta_2) \Big)  \right]  . \]
Choosing $\delta_1 = 1/2$ and $\delta_2 = 2 \sqrt{c_2/c_1}$ we have
\[  h'(x) \in      \left[ -\frac{c_1}{x} \times \Big( 1 + 4 c_2/c_1 \Big)  ,  -\frac{c_1}{x} \times \Big( 1  - 2 \sqrt{c_2/c_1}  \Big)  \right] , \]
which are valid settings since, by assumption,
\[  \delta_2 = 2 \sqrt{c_2/c_1}   <  1 < 2 e^{2c_2/c_1} - 1 .\]
Since $4 c_2/c_1  <    2 \sqrt{c_2/c_1}$ by assumption, the proof is complete.
\end{proof}


\smallskip

\bibliographystyle{alpha}
\bibliography{paper}

\end{document}